\documentclass[12pt,twoside,reqno]{amsart}
\usepackage{amsmath}
\usepackage{amsthm, amscd, amsfonts, amssymb, graphicx}
\usepackage[bookmarksnumbered, plainpages]{hyperref}

\usepackage{mathrsfs}

\allowdisplaybreaks[4]

\newcommand{\tr}{\text{tr}}

\newtheorem{thm}{Theorem}[section]
\newtheorem{cor}[thm]{Corollary}
\newtheorem{lem}[thm]{Lemma}
\newtheorem{prop}[thm]{Proposition}
\newtheorem{defn}[thm]{Definition}

\numberwithin{equation}{section}

\def\tr{\mbox{tr}}
\def\YMH{\mbox{YMH}}
\def\HYM{\mbox{HYM}}

\begin{document}

\title{\bf The Limit of the Yang--Mills--Higgs Flow for twisted Higgs pairs}
\author{Changpeng Pan, Zhenghan Shen and Pan Zhang}

\address{Changpeng Pan\\School of Mathematics and Statistics\\
Nanjing University of Science and Technology\\
Nanjing, 210094,P.R. China\\ }
\email{mathpcp@njust.edu.cn}

\address{Zhenghan Shen\\School of Mathematics and Statistics\\
Nanjing University of Science and Technology\\
Nanjing, 210094,P.R. China\\}\email{mathszh@njust.edu.cn}

\address{Pan Zhang\\School of Mathematical Sciences\\
Anhui University\\
Hefei, 230601, P.R. China\\}\email{panzhang20100@ahu.edu.cn}


\maketitle

\begin{abstract}
In this paper, we consider the Yang--Mills--Higgs flow for twisted Higgs pairs over K\"ahler manifolds. We prove that this flow converges to a reflexive twisted Higgs sheaf outside a closed subset of codimension $4$, and the limiting twisted Higgs sheaf is isomorphic to the double dual of the graded twisted Higgs sheaves associated to the Harder--Narasimhan--Seshadri filtration of the initial twisted Higgs bundle.
\end{abstract}

\vskip 0.2 true cm


\pagestyle{myheadings}
\markboth{\rightline {\scriptsize C. Pan et al.}}
         {\leftline{\scriptsize Yang--Mills--Higgs Flow for Twisted Higgs Pairs}}

\bigskip
\bigskip


\section{ Introduction}
Let $(X,\omega)$ be a compact K\"ahler manifold, $(V,h_{V})$ be a Hermitian holomorphic vector bundle over $X$. A $V$-twisted Higgs bundle is a pair $(E,\phi)$, where $E$ is a holomorphic vector bundle and $\phi:E\rightarrow V\otimes E$ is a holomorphic morphism  satisfying $0=\phi\wedge\phi\in \mbox{End}(E)\otimes\wedge^{2}V$. For untwisted Higgs bundle (i.e. $V=T^{*}X$), it was first studied by Hitchin (\cite{Hit}) on Riemann surface, and by Simpson (\cite{S1,S2}) on K\"ahler manifold.
There are many interesting research about twisted Higgs bundles (see \cite{Bo, GGN, GR, GRa, Ni}, etc.).
Let $H$ be a Hermitian metric on $E$.  Then we can define a dual morphism $\phi^{*H}:E\otimes V\rightarrow E$ by using the Hermitian metrics $H$ and $h_{V}$. Let $[\phi,\phi^{*H}]=\phi\circ\phi^{*H}-\phi^{*H}\circ\phi\in \mbox{End}(E)$.
A Hermitian metric $H$ is said to be  Higgs--Hermitian--Einstein if it satisfies
\begin{equation}
	\sqrt{-1}\Lambda_{\omega}F_{\bar{\partial}_{E},H}+[\phi,\phi^{*H}]=\lambda \cdot \textmd{Id}_{E},
\end{equation}
where $F_{\bar{\partial}_{E},H}$ is the curvature form of the Chern connection $D_{\bar{\partial}_{E},H}$ and $\lambda=\frac{2\pi\mu_{\omega}(E)}{\mbox{Vol}(X,\omega)}$ is a constant. The Donaldson--Uhlenbeck--Yau theorem for twisted Higgs bundles (\cite{AG,BGM,Hit,S1}) guarantees the existence of Higgs--Hermitian--Einstein metrics for the polystable case. It was originally proved by Narasimhan--Seshadri (\cite{NarSe}), Donaldson (\cite{Don1,Don2}) and Uhlenbeck--Yau (\cite{UhYau}) for holomorphic bundles. There are also many interesting and important generalized Donaldson--Uhlenbeck--Yau theorems (see \cite{BLS,BGP,BGPR,BO,iR,LY,MA,Mo,SZZ,WZ,ZZZ}  and references therein).

\medskip

Given a fixed Hermitian metric $H$ on $E$. Let $\mathcal{A}^{1,1}_{H}$ be the space of integrable unitary connections, and $\mathcal{B}_{H}$ be the space of $V$-twisted Higgs pairs. The Yang--Mills--Higgs functional on $\mathcal{B}_{H}$ is defined by
\begin{equation}
	\YMH(A,\phi)=\int_{X}(|F_{A}|^{2}+2|\partial_{A,V}\phi|^{2}+|[\phi,\phi^{*}]|^{2}-2\langle\phi,\phi\rangle_{V})dv_{g},
\end{equation}
where $dv_{g}=\frac{\omega^{n}}{n!}$, $\langle\phi,\phi\rangle_{V}=\tr(\phi \sqrt{-1}\Lambda_{\omega}F_{h_{V}}\phi^{*H})$. The critical point of the Yang--Mills--Higgs functional satisfies
\begin{equation}\label{YMHe}
	\left\{\begin{split}
		&D_{A}^{*}F_{A}+(\partial_{A}-\bar{\partial}_{A})[\phi,\phi^{*H}]=0,\\
		&[\sqrt{-1}\Lambda_{\omega}F_{A}+[\phi,\phi^{*H}],\phi]=0.
	\end{split}\right.
\end{equation}
We say that $(A,\phi)\in\mathcal{B}_{H}$ is a Yang--Mills--Higgs pair if it is a critical point of the Yang--Mills--Higgs functional. By the K\"ahler identities, we know that if $(A,\phi)$ satisfies $\sqrt{-1}\Lambda_{\omega}F_{A}+[\phi,\phi^{*H}]=\lambda \cdot {\rm Id}_{E}$, then it is a Yang--Mills--Higgs pair and $H$ is the Higgs--Hermitian--Einstein metric of $V$-twisted Higgs bundle $(E,\bar{\partial}_{A},\phi)$. The gradient flow of the Yang--Mills--Higgs functional is
\begin{equation}\label{mymh}
	\left\{
	\begin{split}
		\frac{\partial A(t)}{\partial  t}&=-D_{A(t)}^{*}F_{A(t)}-(\partial_{A(t)}-\bar{\partial}_{A(t)})[\phi(t),\phi^{*H}(t)],  \\
		\frac{\partial \phi(t)}{\partial  t}&=-[\sqrt{-1}\Lambda_{\omega}F_{A(t)}+[\phi(t),\phi^{*H}(t)],\phi(t)].
	\end{split}\right.
\end{equation}
The existence of long time solution for the above gradient flow will be discussed in Section \ref{sec:Pre}. If the initial data $(A_{0},\phi_{0})\in\mathcal{B}_{H}$ is stable, the  heat flow converges to a $V$-twisted Higgs pair and the limit must lie in the same orbit of the initial data. In this article, we are interested in the convergence of this flow in the general case.

\medskip

Let $(E,\bar{\partial}_{A},\phi)$ be a $V$-twisted Higgs bundle. There is a filtration of $(E,\bar{\partial}_{A})$ given by $\phi$-invariant subsheaves which is called Harder--Narasimhan--Seshadri (abbr. HNS) filtration. Let $Gr^{HNS}(E,\bar{\partial}_{A},\phi)$ be the associated graded
object (the direct sum of the stable quotients) of the Harder--Narasimhan--Seshadri filtration.
For the holomorphic vector bundle, Atiyah--Bott (\cite{AB}) and Bando--Siu (\cite{BS}) conjectured that there should be a correspondence between the limit of the Yang--Mills flow and the double dual of $Gr^{HNS}(E,\bar{\partial}_{A})$. It was proved by Daskalopoulos (\cite{Da}), Daskalopoulos--Wentworth (\cite{DW}) Jacob (\cite{Ja}) and Sibley (\cite{Si}) in different cases. For the untwisted Higgs bundle, Wilkin (\cite{W}) and Li--Zhang (\cite{LZ,LZ1}) also proved the similar correspondence between Yang--Mills--Higgs flow and the double dual of $Gr^{HNS}(E,\bar{\partial}_{A},\phi)$. In \cite{LZZ}, the authors proved this conjecture for reflexive sheaves. In the meanwhile, Zhang (\cite{Z}) considered this problem for $T^{*}X\otimes L$-twisted Higgs bundles over Riemann surface, and he proved the related correspondence in that case. In the present paper, we extend the above results to $V$-twisted Higgs bundles over K\"ahler manifolds. In fact, we prove the following theorem.

\begin{thm}\label{thm:main1}
	Let $(A(t),\phi(t))$ be a solution of Yang--Mills--Higgs flow (\ref{mymh}) with initial data $(A_{0},\phi_{0})\in\mathcal{B}_{H}$. Then we have:
	\begin{itemize}
		\item[(1)] for every sequence $t_{k}\rightarrow+\infty$, there is a subsequence $t_{k_{j}}$ such that as $t_{k_{j}}\rightarrow+\infty$, $(A(t_{k_{j}}),\phi(t_{k_{j}}))$ converges modulo gauge transformations to a pair $(A_{\infty},\phi_{\infty})$ satisfying (\ref{YMHe}) on the Hermitian vector bundle $(E_{\infty}, H_{\infty})$ in $C_{loc}^{\infty}$ topology outside $\Sigma$, where $\Sigma$ is a closed set of Hausdorff codimension at least 4. The limiting $(E_{\infty}, \bar{\partial}_{A_{\infty}},\phi_{\infty})$ can be extended to the whole $X$ as a reflexive $V$-twisted Higgs sheaf with a holomorphic orthogonal splitting
		\begin{equation}
		(E_{\infty}, H_{\infty}, \bar{\partial}_{A_{\infty}},\phi_{\infty})=\oplus_{i=1}^{l}(E_{\infty}^{i}, H_{\infty}^{i}, \bar{\partial}_{A_{\infty}^{i}},\phi_{\infty}^{i}),
		\end{equation}
	    where $H_{\infty}^{i}$ is an admissible Higgs--Hermitian--Einstein metric on the reflexive $V$-twisted Higgs sheaf $(E_{\infty}^{i}, \bar{\partial}_{A_{\infty}^{i}},\phi_{\infty}^{i})$.
		\item[(2)] let $\{E_{i,j}\}$ be the HNS filtration of the $V$-twisted Higgs bundle $(E, \bar{\partial}_{A_{0}}, \phi_{0})$, the associated graded object $Gr^{HNS}(E,A_{0},\phi_{0})=\oplus _{i=1}^{l}\oplus_{j=1}^{l_{i}} Q_{i,j}$ be uniquely determined by the isomorphism class of $(A_{0}, \phi_{0})$. We have $Gr^{HNS}(E,A_{0},\phi_{0})^{**}\simeq (E_{\infty},\bar{\partial}_{A_{\infty}},\phi_{\infty})$.
	\end{itemize}
\end{thm}

In order to prove the convergence of the flow, we proved the energy inequality, the monotonicity formula of certain quantities and the $\epsilon$-regularity. Following Hong--Tian's arguments (\cite{HT}) and using Bando--Siu's extension technique (\cite{BS}), we obtain the first part of Theorem \ref{thm:main1}. The proof of the second part of Theorem \ref{thm:main1} can be divided into two steps. The first step is to prove that the Harder--Narasimhan (abbr. HN) type of the limiting $V$-twisted Higgs sheaf is in fact equal to that of $(E,A_{0},\phi_{0})$. The second step is to construct a non-zero $\phi$-invariant holomorphic map from $Q_{i,j}$ to the limiting sheaf. The idea of the proof is the same as for untwisted case (\cite{LZ,LZ1}), but there are some differences in the treatment of certain details.

\medskip

This paper is organized as follows. In Section \ref{sec:Pre}, we build some basic estimates for Donaldson heat flow and the Yang--Mills--Higgs flow for twisted Higgs pairs. In Section \ref{sec:Con}, we prove the monotonicity inequality and the $\epsilon$-regularity estimate for the Yang--Mills--Higgs flow and complete the first part of Theorem \ref{thm:main1}.  In Section \ref{sec:Is}, we first prove that the HN type of the limiting twisted Higgs sheaf is in fact equal to the type of the initial twisted Higgs bundle, and then complete the proof of the second part of Theorem \ref{thm:main1}.

\section{Preliminary}\label{sec:Pre}
Let $(X,\omega)$ be an $n$-dimensional compact K\"ahler manifold, and $(E,\phi)$ be a $V$-twisted Higgs vector bundle over $X$. We consider the following Donaldson heat flow
\begin{equation}\label{DF}
	\left\{\begin{split}
		&H^{-1}(t)\frac{\partial H(t)}{\partial t}=-2(\sqrt{-1}\Lambda_{\omega}F_{H(t)}+[\phi,\phi^{*H(t)}]-\lambda \cdot \textmd{Id}_{E}),\\
		&H(0)=H_{0}.
	\end{split}\right.
\end{equation}
It is a  strictly parabolic equation, so the standard parabolic theory gives the short time existence. And the long time existence can be proved by the  same method of Simpson's article (\cite{S1}). The following lemma can be obtained by direct calculation.

\begin{lem}\label{ll:1}
	Let $H(t)$ be the solution of the flow (\ref{DF}) and set $\Phi(H(t))=\sqrt{-1}\Lambda_{\omega}F_{H(t)}+[\phi,\phi^{*H(t)}]-\lambda \cdot \textmd{Id}_{E}$, then
	\begin{equation}\label{eq2:2}
		\bigg(\frac{\partial}{\partial t}-\Delta\bigg){\rm tr}(\Phi(H(t)))=0,
	\end{equation}
and
	\begin{equation}
		\bigg(\frac{\partial}{\partial t}-\Delta\bigg)|\Phi(H(t))|^{2}_{H(t)}=-4|\bar{\partial}_{E}\Phi(H(t))|^{2}_{\omega,H(t)}-4|[\phi,\Phi(H(t))]|^{2}_{h_{V},H(t)}.
	\end{equation}
\end{lem}

Let $(E,H_{0})$ be a Hermitian vector bundle, and $\mathcal{B}_{H_{0}}$ be the space of $V$-twisted Higgs pairs. Let $\mathcal{G}^{\mathbb{C}}$ (resp. $\mathcal{G}$) be the complex gauge group (resp. unitary gauge group) of $(E,H_{0})$. The complex gauge group $\mathcal{G}^{\mathbb{C}}$ acts on $\mathcal{B}_{H_{0}}$ as follows
\begin{equation}
	\begin{split}
		\sigma\cdot (\bar{\partial}_{A},\phi)=(\sigma\circ \bar{\partial}_{A}\circ\sigma^{-1},\sigma\circ \phi\circ\sigma^{-1}),\ \ \ \ \forall\sigma\in\mathcal{G}^{\mathbb{C}},\ (A,\phi)\in\mathcal{B}_{H_{0}}.
	\end{split}
\end{equation}
Following the methods in \cite{Don1,LZ}, we have the following proposition.
\begin{prop}
	There is a family of complex gauge transformations $\sigma(t)\in\mathcal{G}^{\mathbb{C}}$  such that $(A(t),\phi(t))=\sigma(t)\cdot(A_{0},\phi_{0})$ is a long time solution of the Yang--Mills--Higgs flow (\ref{mymh}) with the initial data $(A_{0},\phi_{0})$, where $\sigma^{*H_{0}}(t)\sigma(t)=H_{0}^{-1}H(t)$ and $H(t)$ is the long time solution of Donaldson heat flow (\ref{DF}) for $(E,\bar{\partial}_{A_{0}},\phi_{0})$.
\end{prop}

Along the Yang--Mills--Higgs flow, we have the following energy identity.
\begin{prop}\label{p:1}
	Let $(A(t), \phi(t))$ be a solution of the Yang--Mills--Higgs flow (\ref{mymh}) with initial twisted Higgs pair $(A_{0}, \phi_{0})$. Then
	\begin{equation}\label{Prop:1eq}
		\YMH (A(t),\phi(t))+2\int_{0}^{t}\int_{X}\bigg(\bigg|\frac{\partial A}{\partial t}\bigg|^{2}+\bigg|\frac{\partial \phi}{\partial t}\bigg|^{2}+\bigg|\frac{\partial \phi^{*}}{\partial t}\bigg|^{2}\bigg)dv_{g}dt=\YMH(A_{0},\phi_{0}).
	\end{equation}
\end{prop}
\begin{proof}
	Let $(A_{t},\phi_{t})$ be a family of twisted Higgs pairs and $\frac{d}{dt}\Big|_{t=0}(A_{t},\phi_{t})=(\dot{A},\dot{\phi})$, then
	\begin{equation*}
		\begin{split}
			\frac{d}{dt}\Big|_{t=0}\YMH(A_{t},\phi_{t})=&2\Re\int_{X}(\langle D_{A}\dot{A},F_{A}\rangle+2\langle[\dot{A}^{1,0},\phi]+\partial_{A,V}\dot{\phi},\partial_{A,V}\phi\rangle\\
			&+\langle[\dot{\phi},\phi^{*}]+[\phi,\dot{\phi}^{*}],[\phi,\phi^{*}]\rangle-2\langle \dot{\phi},\phi\rangle_{V}
			)dv_{g},
		\end{split}
	\end{equation*}
	where
	\begin{equation*}
		\langle[\dot{A}^{1,0},\phi],\partial_{A,V}\phi\rangle=\langle \dot{A}^{1,0},[\partial_{A,V}\phi,\phi^{*}]\rangle=-\overline{\langle \dot{A}^{0,1},[\phi,\bar{\partial}_{A,V^{*}}\phi^{*}]\rangle},
	\end{equation*}
	\begin{equation*}
		\begin{split}
			\langle \dot{\phi},\partial_{A,V}^{*}\partial_{A,V}\phi\rangle-\langle\dot{\phi},\phi\rangle_{V}
			=\langle \dot{\phi},[\sqrt{-1}\Lambda_{\omega}F_{A},\phi]\rangle=-\overline{\langle \dot{\phi}^{*},[\sqrt{-1}\Lambda_{\omega}F_{A},\phi^{*}]\rangle},
		\end{split}
	\end{equation*}
and
	\begin{equation*}
		\begin{split}
			\langle[\dot{\phi},\phi^{*}],[\phi,\phi^{*}]\rangle=\langle\dot{\phi},[[\phi,\phi^{*}],\phi]\rangle,\ \ \langle[\phi,\dot{\phi}^{*}],[\phi,\phi^{*}]\rangle=-\langle\dot{\phi}^{*},[[\phi,\phi^{*}],\phi^{*}]\rangle.
		\end{split}
	\end{equation*}
	Since $\bar{\partial}_{A,V}\phi=0$ and $\partial_{A,V^{*}}\phi^{*}=0$, we have
	\begin{equation*}
		\begin{split}
			\frac{d}{dt}\Big|_{t=0}\YMH(A_{t},\phi_{t})=&2\int_{X}(\langle \dot{A},D_{A}^{*}F_{A}\rangle+\langle\dot{A},(\partial_{A}-\bar{\partial}_{A})[\phi,\phi^{*}]\rangle\\
			&+\langle\dot{\phi},[\sqrt{-1}\Lambda_{\omega}F_{A}+[\phi,\phi^{*}],\phi]\rangle-\langle\dot{\phi}^{*},[\sqrt{-1}\Lambda_{\omega}F_{A}+[\phi,\phi^{*}],\phi^{*}]\rangle
			)dv_{g}.
		\end{split}
	\end{equation*}
	Using the Yang--Mills--Higgs flow equation (\ref{mymh}), we have
	\begin{equation}\label{Prop:2eq}
		\begin{split}
			\frac{d}{dt}\YMH(A(t),\phi(t))=-2\int_{X}\bigg(\bigg|\frac{\partial A}{\partial t}\bigg|^{2}+\bigg|\frac{\partial \phi}{\partial t}\bigg|^{2}+\bigg|\frac{\partial \phi^{*}}{\partial t}\bigg|^{2}\bigg)dv_{g}.
		\end{split}
	\end{equation}
	Integrating   the above equality (\ref{Prop:2eq}) from $0$ to $t$ gives (\ref{Prop:1eq}).
\end{proof}

\begin{prop}\label{ppp:1} Let $(A(t),\phi(t))$ be a solution of the heat flow (\ref{mymh}) with initial twisted Higgs pair $(A_{0}, \phi_{0})$. Then
	\begin{equation}
		\begin{split}
			\bigg(\frac{\partial }{\partial t}-\Delta\bigg)|\phi|^{2}_{H_{0},h_{V}}\leq&-2|\partial_{A,V}\phi|^{2}-C_{3}(|\phi|^{2}+1)^{2}+C_{4}(|\phi|^{2}+1),
		\end{split}
	\end{equation}
	where $C_{3}$, $C_{4}$  are constants depending  on $\sup_{X}|\sqrt{-1}\Lambda_{\omega}F_{h_{V}}|$. Moreover, we have
	\begin{equation}
		\sup_{X}|\phi|^{2}\leq \max\{\sup_{X}|\phi_{0}|^{2},C_{4}/C_{3}\}.
	\end{equation}
\end{prop}

\begin{proof}
By direct calculation, there holds
	\begin{equation*}
		\begin{split}
			\bigg(\frac{\partial}{\partial t}-\Delta\bigg)|\phi|^{2}_{H_{0},h_{V}}=&-2|\partial_{A,V}\phi|^{2}-2\langle\phi,[[\phi,\phi^{*}],\phi]\rangle+2\langle\phi,\sqrt{-1}\Lambda_{\omega}F_{h_{V}}(\phi)\rangle\\
			=&-2|\partial_{A,V}\phi|^{2}-2|[\phi,\phi^{*}]|^{2}+2\langle\phi,\sqrt{-1}\Lambda_{\omega}F_{h_{V}}(\phi)\rangle.
		\end{split}
	\end{equation*}
Let $\{v^{i}\}$ be a local orthonormal frame of $V$ and $\phi=\phi_{i}v^{i}$. Since $\phi\wedge\phi=0$, we have $\phi_{i}\phi_{j}=\phi_{j}\phi_{i}$ and
	\begin{equation*}
		\begin{split}
			|[\phi,\phi^{*}]|^{2}=&\sum_{i,j}\tr([\phi_{i},\phi_{i}^{*}][\phi_{j},\phi_{j}^{*}])=\sum_{i,j}\tr([\phi_{i},\phi_{j}^{*}][\phi_{j},\phi_{i}^{*}])\\
			=&\sum_{i,j}|[\phi_{i},\phi_{j}^{*}]|^{2}\geq \sum_{i}|[\phi_{i},\phi_{i}^{*}]|^{2}.
		\end{split}
	\end{equation*}
	According to the Lemma 2.7 in \cite{S2}, we have
	\begin{equation*}
		|[\phi_{i},\phi_{i}]|^{2}\geq C_{1}(|\phi_{i}|^{2}+1)^{2}-C_{2}(|\phi_{i}|^{2}+1).
	\end{equation*}
The above calculation leads to the proof.
\end{proof}

In the following, we derive the local energy monotonic inequality along the flow. Let $e_{2}(A,\phi)=|F_{A}|^{2}+2|\partial_{A,V}\phi|^{2}$ and $f\in C^{\infty}(X)$, then
\begin{equation*}
	\begin{split}
		\frac{d}{dt}\int_{X}f^{2}e_{2}(A,\phi)dv_{g}=&2\Re\int_{X}\bigg(\bigg\langle \frac{\partial A}{\partial t},D_{A}^{*}(f^{2}F_{A})\bigg\rangle+f^{2}\bigg\langle\frac{\partial A}{\partial t},(\partial_{A}-\bar{\partial}_{A})[\phi,\phi^{*}]\bigg\rangle\\
		&+2\bigg\langle\frac{\partial\phi}{\partial t},\partial_{A,V}^{*}(f^{2}\partial_{A,V}\phi) \bigg\rangle\bigg)dv_{g},
	\end{split}
\end{equation*}
where
\begin{equation*}
	\begin{split}
		D^{*}_{A}(f^{2}F_{A})=&\sqrt{-1}[\Lambda_{\omega},\bar{\partial}_{A}-\partial_{A}](f^{2}F_{A})\\
		=&\sqrt{-1}\Lambda_{\omega}(\bar{\partial}-\partial)(f^{2})\wedge F_{A}-f^{2}D_{A}^{*}F_{A}-(\bar{\partial}-\partial)(f^{2})\sqrt{-1}\Lambda_{\omega}F_{A},
	\end{split}
\end{equation*}
and
\begin{equation*}
	\begin{split}
		\partial_{A,V}^{*}(f^{2}\partial_{A,V}\phi)=\sqrt{-1}\Lambda_{\omega}\bar{\partial}(f^{2})\wedge\partial_{A,V}\phi+f^{2}\partial_{A,V}^{*}\partial_{A,V}\phi.
	\end{split}
\end{equation*}
Therefore,
\begin{equation*}
	\begin{split}
		&\frac{d}{dt}\int_{X}f^{2}e_{2}(A,\phi)dv_{g}\\
&=2\int_{X}f^{2}\bigg(\bigg\langle \frac{\partial A}{\partial t},D_{A}^{*}F_{A}\bigg\rangle+\bigg\langle\frac{\partial A}{\partial t},(\partial_{A}-\bar{\partial}_{A})[\phi,\phi^{*}]\bigg\rangle\\
		&+\bigg\langle \frac{\partial \phi}{\partial t},[\sqrt{-1}\Lambda_{\omega}F_{A},\phi]\bigg\rangle-\bigg\langle\frac{\partial\phi^{*}}{\partial t},[\sqrt{-1}\Lambda_{\omega}F_{A},\phi^{*}]\bigg\rangle\bigg)dv_{g}\\
		&+2\Re\int_{X}\bigg(\bigg\langle\frac{\partial A}{\partial t},\sqrt{-1}\Lambda_{\omega}(\bar{\partial}-\partial)(f^{2})\wedge F_{A}\bigg\rangle-\bigg\langle\frac{\partial A}{\partial t},\sqrt{-1}(\bar{\partial}-\partial)(f^{2})\Lambda_{\omega} F_{A}\bigg\rangle\\
		&+2\bigg\langle\frac{\partial\phi}{\partial  t},\sqrt{-1}\Lambda_{\omega}\bar{\partial}(f^{2})\wedge\partial_{A,V}\phi \bigg\rangle+2f^{2}\bigg\langle \frac{\partial \phi}{\partial t},\phi\bigg\rangle\bigg)dv_{g}.
	\end{split}
\end{equation*}
Let $f$ be a cut-off function on $B_{2R}(x_{0})$, satisfy $0\leq f\leq 1$, $f\equiv 1$ on $B_{R}(x_{0})$ and $|df|\leq\frac{2}{R}$. Then
\begin{equation*}
	\begin{split}
		&\bigg|\frac{d}{dt}\int_{X}f^{2}e_{2}(A,\phi)dv_{g}+2\int_{X}f^{2}\bigg(\bigg| \frac{\partial A}{\partial t}\bigg|^{2}+\bigg| \frac{\partial \phi}{\partial t}\bigg|^{2}+\bigg|\frac{\partial \phi^{*}}{\partial t}\bigg|^{2}\bigg)dv_{g}\bigg|\\
		\leq&\frac{C_{1}}{R}\Big(\int_{X}\bigg(\bigg|\frac{\partial A}{\partial t}\bigg|^{2}+\bigg|\frac{\partial \phi}{\partial t}\bigg|^{2}+\bigg|\frac{\partial \phi^{*}}{\partial t}\bigg|^{2}\bigg)dv_{g}\Big)^{1/2}+C_{2}\Big(\int_{X}\bigg(\bigg|\frac{\partial A}{\partial t}\bigg|^{2}+\bigg|\frac{\partial \phi}{\partial t}\bigg|^{2}+\bigg|\frac{\partial \phi^{*}}{\partial t}\bigg|^{2}\bigg)dv_{g}\Big)^{1/2},
	\end{split}
\end{equation*}
where $C_{1}$, $C_{2}$ are constants depending on $\sup_{X}|\sqrt{-1}\Lambda_{\omega}F_{h_{V}}|$, $\sup_{X}|\phi_{0}|$, $\YMH(A_{0},\phi_{0})$ and the geometry of $(X,\omega)$. Then we have the following proposition.

\begin{prop}\label{p:3}
	Let $(A(t),\phi(t))$ be a solution of the  Yang--Mills--Higgs flow (\ref{mymh}), then for any $B_{2R}(x_{0})\subset X$, $s,\tau$, we have
	\begin{equation*}
		\begin{split}
			\int_{B_{R}(x_{0})}&e_{2}(A,\phi)(\cdot, s)dv_{g}\\
			\leq &\int_{B_{2R}(x_{0})}e_{2}(A,\phi)(\cdot,\tau)dv_{g}+2\int_{\min\{s,\tau\}}^{\max\{s,\tau\}}\int_{X}\bigg(\bigg| \frac{\partial A}{\partial t}\bigg|^{2}+\bigg| \frac{\partial \phi}{\partial t}\bigg|^{2}+\bigg|\frac{\partial \phi^{*}}{\partial t}\bigg|^{2}\bigg)dv_{g}dt\\
			&+C_{1}\bigg\{\frac{|s-\tau|}{R^{2}}\int_{\min\{s,\tau\}}^{\max\{s,\tau\}}\int_{X}\bigg(\bigg| \frac{\partial A}{\partial t}\bigg|^{2}+\bigg| \frac{\partial \phi}{\partial t}\bigg|^{2}+\bigg|\frac{\partial\phi^{*}}{\partial t}\bigg|^{2}\bigg)dv_{g}dt\bigg\}^{1/2}\\
			&+C_{2}\bigg\{|s-\tau|\int_{\min\{s,\tau\}}^{\max\{s,\tau\}}\int_{X}\bigg(\bigg| \frac{\partial A}{\partial t}\bigg|^{2}+\bigg| \frac{\partial \phi}{\partial t}\bigg|^{2}+\bigg|\frac{\partial \phi^{*}}{\partial t}\bigg|^{2}\bigg)dv_{g}dt\bigg\}^{1/2},
		\end{split}
	\end{equation*}
	where $C_{1}, C_{2}$ are constants depending on $\sup_{X}|\sqrt{-1}\Lambda_{\omega}F_{h_{V}}|$, $\sup_{X}|\phi_{0}|$, $\YMH(A_{0},\phi_{0})$ and the geometry of $(X,\omega)$.
\end{prop}

\medskip

Let $\theta(t)=\sqrt{-1}\Lambda_{\omega}F_{A(t)}+[\phi(t),\phi^{*}(t)]$ and
 $$I(t)=\int_{X}(|D_{A(t)}\theta(t)|^{2}_{\omega,H_{0}}+2|[\phi(t),\theta(t)]|_{h_{V},H_{0}}^{2})dv_{g}.$$
Then we have the following proposition:
\begin{prop}
	\begin{equation}
		I(t)\rightarrow 0,\ \ as\ t\rightarrow+\infty.
	\end{equation}
\end{prop}
\begin{proof}
The proof is exactly the same as untwisted case (\cite{LZ}), so we omit here.
\end{proof}

Let $\nabla_{A,V}$ be the induced connection on $\Omega^{*}_{X}\otimes\mbox{End}(E)\otimes V$ induced by $D_{A},D_{h_{V}}$ and the Chern connection on $TM$ and $\nabla_{A}$ be the covariant derivative corresponding to $D_{A}$.

\begin{prop}\label{p:2}
	Along the Yang--Mills--Higgs flow (\ref{mymh}), we have
	\begin{equation}\label{eee:1}
		\begin{split}
			\bigg(\frac{\partial}{\partial t}-\Delta\bigg)&|\partial_{A,V}\phi|^{2}+2|\nabla_{A,V}\partial_{A,V}\phi|^{2}\\
			\leq&C_{1}(|\phi|^{2}+|F_{A}|+|F_{h_V}|+|Ric|)|\partial_{A,V}\phi|^{2}+C_{2}|\partial_{A}\sqrt{-1}\Lambda_{\omega}F_{h_{V}}||\phi||\partial_{A,V}\phi|
		\end{split}
	\end{equation}
and
\begin{equation}\label{eee:2}
	\begin{split}
		\bigg(\frac{\partial}{\partial t}-\Delta\bigg)&|F_{A}|^{2}+|\nabla_{A}F_{A}|^{2}\\
		\leq& C_{3}(|F_{A}|+|\phi|^{2}+|Rm|)|F_{A}|^{2}+C_{4}|F_{A}||\partial_{A,V}\phi|^{2}+C_{5}|\phi|^{2}|F_{h_{V}}||F_{A}|,
	\end{split}
\end{equation}
	where the constants $C_i(i=1,\cdots,5)$ are depending only on the dimension $n$.
\end{prop}
\begin{proof}
In local normal coordinates, we have
\begin{equation*}
	\begin{split}
		\Delta|\partial_{A,V}\phi|^{2}=2|\nabla_{A,V}\partial_{A,V}\phi|^{2}+2\langle\nabla_{\alpha}\nabla_{\bar{\alpha}}\partial_{A,V}\phi,\partial_{A,V}\phi\rangle+2\langle\partial_{A,V}\phi,\nabla_{\bar{\alpha}}\nabla_{\alpha}\partial_{A,V}\phi\rangle,
	\end{split}
\end{equation*}
where
\begin{equation*}
	\begin{split}
		\nabla_{\alpha}\nabla_{\bar{\alpha}}\partial_{A,V}\phi=&\nabla_{\alpha}\nabla_{\bar{\alpha}}\nabla_{\beta}\phi dz^{\beta}\\
		=&-\nabla_{\alpha}(F_{A,V;\beta\bar{\alpha}}\phi)dz^{\beta}\\
		=&-\nabla_{\alpha}(F_{A,V;\beta\bar{\alpha}})\phi dz^{\beta}-F_{A,V;\beta\bar{\alpha}}\nabla_{\alpha}\phi dz^{\beta}\\
		=&-\nabla_{\beta}(F_{A,V;\alpha\bar{\alpha}})\phi dz^{\beta}-F_{A,V;\beta,\bar{\alpha}}\nabla_{\alpha}\phi dz^{\beta}\\
		=&-\partial_{A}(\sqrt{-1}\Lambda_{\omega}F_{A,V})\phi-F_{A,V;\beta,\bar{\alpha}}\nabla_{\alpha}\phi dz^{\beta}\\
		=&-[\partial_{A}(\sqrt{-1}\Lambda_{\omega}F_{A}),\phi]-\partial_{A}(\sqrt{-1}\Lambda_{\omega}F_{h_V})\phi-F_{A,V;\beta,\bar{\alpha}}\nabla_{\alpha}\phi dz^{\beta},
	\end{split}
\end{equation*}
\begin{equation*}
	\begin{split}
		\nabla_{\bar{\alpha}}\nabla_{\alpha}\partial_{A,V}\phi=&\nabla_{\bar{\alpha}}\nabla_{\alpha}\nabla_{\beta}\phi dz^{\beta}+\nabla_{\beta}\phi\nabla_{\bar{\alpha}}\nabla_{\alpha}dz^{\beta}\\
		=&\nabla_{\bar{\alpha}}\nabla_{\beta}\nabla_{\alpha}\phi dz^{\beta}+\nabla_{\beta}\phi\nabla_{\bar{\alpha}}\nabla_{\alpha}dz^{\beta}\\
		=&\nabla_{\beta}\nabla_{\bar{\alpha}}\nabla_{\alpha}\phi dz^{\beta}-F_{A,V;\beta\bar{\alpha}}\nabla_{\alpha}\phi dz^{\beta}+\nabla_{\beta}\phi\nabla_{\bar{\alpha}}\nabla_{\alpha}dz^{\beta}\\
		=&-\partial_{A}(\sqrt{-1}\Lambda_{\omega}F_{A,V}\phi)-F_{A,V;\beta\bar{\alpha}}\nabla_{\alpha}\phi dz^{\beta}+\nabla_{\beta}\phi\nabla_{\bar{\alpha}}\nabla_{\alpha}dz^{\beta}\\
		=&-\partial_{A}([\sqrt{-1}\Lambda_{\omega}F_{A},\phi]+\sqrt{-1}\Lambda_{\omega}F_{h_V}\phi)-F_{A,V;\beta\bar{\alpha}}\nabla_{\alpha}\phi dz^{\beta}+\nabla_{\beta}\phi\nabla_{\bar{\alpha}}\nabla_{\alpha}dz^{\beta}\\
		=&-[\partial_{A}\sqrt{-1}\Lambda_{\omega}F_{A},\phi]+[\sqrt{-1}\Lambda_{\omega}F_{A},\partial_{A}\phi]+(\partial_{A}\sqrt{-1}\Lambda_{\omega}F_{h_V})\phi\\
		&+\sqrt{-1}\Lambda_{\omega}F_{h_V}\partial_{A}\phi-F_{A,V;\beta\bar{\alpha}}\nabla_{\alpha}\phi dz^{\beta}+\nabla_{\beta}\phi\nabla_{\bar{\alpha}}\nabla_{\alpha}dz^{\beta}.\\
	\end{split}
\end{equation*}	
	On the other hand, using the heat flow equations (\ref{mymh}), we have
\begin{equation*}
	\begin{split}
		\frac{\partial}{\partial t}|\partial_{A,V}\phi|^{2}_{\omega,h_{V},H_{0}}=&2\Re\bigg(\bigg\langle\bigg[\frac{\partial A^{1,0}}{\partial t},\phi\bigg],\partial_{A,V}\phi\bigg\rangle+\bigg\langle\partial_{A,V}\frac{\partial \phi}{\partial t},\partial_{A,V}\phi\bigg\rangle\bigg)\\
		=&-2\Re\big(\langle[\partial_{A}(\sqrt{-1}\Lambda_{\omega}F_{A}+[\phi,\phi^{*}]),\phi],\partial_{A,V}\phi\rangle\\
		&+\langle\partial_{A,V}[\sqrt{-1}\Lambda_{\omega}F_{A}+[\phi,\phi^{*}],\phi],\partial_{A,V}\phi\rangle\big)\\
		=&-2\Re\big(2\langle[\partial_{A}(\sqrt{-1}\Lambda_{\omega}F_{A}+[\phi,\phi^{*}]),\phi],\partial_{A,V}\phi\rangle\\
		&+\langle[\sqrt{-1}\Lambda_{\omega}F_{A}+[\phi,\phi^{*}],\partial_{A,V}\phi],\partial_{A,V}\phi\rangle\big).
	\end{split}
\end{equation*}
Then (\ref{eee:1}) follows from the above identities. The proof of the other equation (\ref{eee:2}) is similar and we omit here.
\end{proof}

\section{Convergence of the Yang--Mills--Higgs flow for twisted Higgs pairs}\label{sec:Con}

In this section, we consider the convergence of the Yang--Mills--Higgs flow (\ref{mymh}) for twisted Higgs pairs on the Hermitian bundle $(E,H_{0})$. We first prove the monotonicity inequality and the $\epsilon$-regularity theorem for the flow.  We will adapt the same arguments used in studying the Yang--Mills flow (\cite{CShen1, CShen2}) and the Yang--Mills--Higgs flow for untwisted Higgs pairs (\cite{LZ}) to the Yang--Mills--Higgs flow for twisted Higgs pairs.

 Let $u=(x,t)\in X\times\mathbb{R}$. For any $u_{0}=(x_{0},t_{0})\in X\times\mathbb{R}^{+}$, set
\begin{equation*}
	\begin{split}
		&S_{r}(u_{0})=X\times\{t=t_{0}-r^{2}\},\\
		&T_{r}(u_{0})=X\times [t_{0}-4r^{2},t_{0}-r^{2}],\\
		&P_{r}(u_{0})=B_{r}(x_{0})\times [t_{0}-r^{2},t_{0}+r^{2}].
	\end{split}
\end{equation*}
For simplicity, we denote $S_{r}(0,0),T_{r}(0,0),P_{r}(0,0)$ by $S_{r},T_{r},P_{r}$.

The fundamental solution of (backward) heat equation with singularity at $u_{0}=(x_{0}, t_{0})$ is
\begin{equation*}
	G_{u_{0}}(x,t)=G_{(x_{0},t_{0})}(x,t)=\frac{1}{(4\pi(t_{0}-t))^{n}}\exp\Big(-\frac{|x-x_{0}|^{2}}{4(t_{0}-t)}\Big),\ \ t\leq t_{0}.
\end{equation*}
For simplicity, denote $G_{(0,0)}(x,t)$ by $G(x,t)$.

Given $0<R\leq i_{X}$, we take $f\in\ C^{\infty}_{0}(B_{R})$ satisfying $0\leq f\leq 1$, $f\equiv 1$ on $B_{R/2}$ and $|\nabla f|\leq 2/R$ on $B_{R}\setminus B_{R/2}$. Let $(A(t),\phi(t))$ be a solution of the Yang--Mills--Higgs flow with initial value $(A_{0},\phi_{0})$ and set
\begin{equation*}
	e_{2}(A,\phi)=|F_{A}|^{2}+2|\partial_{A,V}\phi|^{2},
\end{equation*}

\begin{equation*}
	\Phi(r)=r^{2}\int_{T_{r}(u_{0})}e_{2}(A,\phi)f^{2}G_{u_{0}}dv_{g}dt.
\end{equation*}
Then we have the following theorem.

\begin{thm}\label{thm:4}
	Let $(A(t),\phi(t))$ be a solution of the Yang--Mills--Higgs flow (\ref{mymh}) with initial value $(A_0,\phi_0)$. For any $u_{0}=(x_{0},t_{0})\in X\times [0,T]$ and $0<r_{1}\leq r_{2}\leq \min\{R/2,\sqrt{t_{0}}/2\}$, we have
	\begin{equation}
		\begin{split}
			\Phi(r_{1})\leq &C\exp(C(r_{2}-r_{1}))\Phi(r_{2})+C(r_{2}^{2}-r_{1}^{2})\\
			&+CR^{2-2n}\int_{P_{R}(u_{0})}e_{2}(A,\phi)dv_{g}dt,
		\end{split}
	\end{equation}
	where the constant $C$ depends only on the geometry of $(X,\omega)$, $\sup_{X}|\sqrt{-1}\Lambda_{\omega}F_{h_{V}}|$ and the initial data $(A_{0},\phi_{0})$.
\end{thm}
\begin{proof}
	Choosing normal geodesic coordinates $\{x^{i}\}_{i=1}^{2n}$ in the geodesic ball $B_{R}(x_{0})$, then it follows that
	\begin{equation}\label{a}
		|g_{ij}(x)-\delta_{ij}|\leq C|x|^{2},\ \ |\partial_{k}g_{ij}(x)|\leq C|x|,\ \ \forall x\in B_{r}(x_0),
	\end{equation}
	where $C$ is a positive constant depending only on $x_{0}$.
	
Let $x=r\tilde{x}$, $t=t_{0}+r^{2}\tilde{t}$. There holds that
\begin{equation*}
	\begin{split}
		\Phi(r)&=r^{2}\int_{T_{r}(u_{0})}e_{2}(A,\phi)f^{2}G_{u_{0}}dv_{g}dt\\
		&=r^{2}\int_{t_{0}-4r^{2}}^{t_{0}-r^{2}}\int_{\mathbb{R}^{2n}}e_{2}(A,\phi)(x,t)f^{2}(x)G_{u_{0}}(x,t)\sqrt{\det{(g_{ij})}}(x)dxdt\\
		&=r^{4}\int_{-4}^{-1}\int_{\mathbb{R}^{2n}}e_{2}(A,\phi)(r\tilde{x},t_{0}+r^{2}\tilde{t})f^{2}(r\tilde{x})G(\tilde{x},\tilde{t})\sqrt{\det{(g_{ij})}}(r\tilde{x})d\tilde{x}d\tilde{t}.
	\end{split}
\end{equation*}
Then one can see that
\begin{equation*}
	\begin{split}
		&\frac{d\Phi(r)}{dr}=4r^{3}\int_{-4}^{-1}\int_{\mathbb{R}^{2n}}e_{2}(A,\phi)(r\tilde{x},t_{0}+r^{2}\tilde{t})f^{2}(r\tilde{x})G(\tilde{x},\tilde{t})\sqrt{\det{(g_{ij})}}(r\tilde{x})d\tilde{x}d\tilde{t}\\
		&+r^{3}\int_{-4}^{-1}\int_{\mathbb{R}^{2n}}\{x^{i}\partial_{i}e_{2}(A,\phi)(r\tilde{x},t_{0}+r^{2}\tilde{t})\}f^{2}(r\tilde{x})G(\tilde{x},\tilde{t})\sqrt{\det{(g_{ij})}}(r\tilde{x})d\tilde{x}d\tilde{t}\\
		&+r^{3}\int_{-4}^{-1}\int_{\mathbb{R}^{2n}}\{2(t-t_{0})\partial_{t}e_{2}(A,\phi)(r\tilde{x},t_{0}+r^{2}\tilde{t})\}f^{2}(r\tilde{x})G(\tilde{x},\tilde{t})\sqrt{\det{(g_{ij})}}(r\tilde{x})d\tilde{x}d\tilde{t}\\
		&+r^{4}\int_{-4}^{-1}\int_{\mathbb{R}^{2n}}e_{2}(A,\phi)(r\tilde{x},t_{0}+r^{2}\tilde{t})\frac{d}{dr}\{f^{2}(r\tilde{x})\sqrt{\det{(g_{ij})}}(r\tilde{x})\}G(\tilde{x},\tilde{t})d\tilde{x}d\tilde{t}\\
		=&I_{1}+I_{2}+I_{3}+I_{4}.
	\end{split}
\end{equation*}
For the second term $I_2$, we have
\begin{equation*}
	\begin{split}
		I_{2}=&r\int_{T_{r}(u_{0})}\{x^{i}\partial_{i}e_{2}(A,\phi)(x,t)\}f^{2}(x)G_{u_{0}}(x,t)\sqrt{\det{(g_{ij})}}(x)dxdt,\\
	\end{split}
\end{equation*}
where
\begin{equation*}
	\begin{split}
		x^{i}\partial_{i}e_{2}(A,\phi)=&2\Re(\langle x^{i}\nabla_{i} F_{A},F_{A}\rangle+2\langle x^{i}\nabla_{i}\partial_{A,V}\phi,\partial_{A,V}\phi\rangle).
	\end{split}
\end{equation*}
 By the Bianchi identity, we have
\begin{equation*}
	\begin{split}
		2\langle x^{i}\nabla_{i} F_{A},F_{A}\rangle=&\langle x^{i}\nabla_{i} F_{A}(\partial_{j},\partial_{k})dx^{j}\wedge dx^{k},F_{A}\rangle\\
		=&\langle x^{i}\nabla_{j} F_{A}(\partial_{i},\partial_{k})dx^{j}\wedge dx^{k},F_{A}\rangle+\langle x^{i}\nabla_{k} F_{A}(\partial_{j},\partial_{i})dx^{j}\wedge dz^{k},F_{A}\rangle\\
		=&2\langle x^{i}D_{A}(F_{A,ik}dx^{k})-x^{i}(F_{A}(\nabla_{j}\partial_{i},\partial_{k})+F_{A}(\partial_{i},\nabla_{j}\partial_{k}))dx^{j}\wedge dx^{k},F_{A}\rangle\\
		=&2\langle D_{A}(x^{i}F_{A,ik}dx^{k})-x^{i}F_{A}(\nabla_{j}\partial_{i},\partial_{k})dx^{j}\wedge dx^{k},F_{A}\rangle-4|F_{A}|^{2}.
	\end{split}
\end{equation*}
Set
$$x\odot F_{A}=\frac{1}{2}x^{i}F_{A,ij}dx^{j},$$
we have
\begin{equation}\label{eqn:35}
	\begin{split}
		\Re\langle x\odot F_{A},D_{A}^{*}F_{A}\rangle=-\Re\langle x\odot F_{A},\frac{\partial A}{\partial t}\rangle
		+\Re\langle x\odot F_{A},(\bar{\partial}_{A}-\partial_{A})[\phi,\phi^{*}]\rangle.
	\end{split}
\end{equation}
In addition,
\begin{equation*}
	\begin{split}
		\langle x^{i}\nabla_{i}\nabla_{A,V}\phi,\nabla_{A,V}\phi\rangle=&\langle x^{i}\nabla_{i}(\nabla_{j}\phi dx^{j}),\nabla_{A,V}\phi\rangle\\
		=&\langle x^{i}\nabla_{j}\nabla_{i}\phi dx^{j},\nabla_{A,V}\phi\rangle+\langle x^{i}F_{A,V;ij}\phi dx^{j},\nabla_{A,V}\phi\rangle\\
		&+\langle x^{i}\nabla_{j}\phi \nabla_{i}dx^{j}),\nabla_{A,V}\phi \rangle\\
		=&\langle D_{A,V}(x^{i}\nabla_{i}\phi),\nabla_{A,V}\phi\rangle-\langle x^{i}\nabla_{A,V}\phi (\nabla_{i}\partial_{j})dx^{j}),\nabla_{A,V}\phi \rangle\\
		&+\langle x^{i}F_{A,V;ij}\phi dx^{j},\nabla_{A,V}\phi\rangle-|\nabla_{A,V}\phi|^{2},
	\end{split}
\end{equation*}
and
\begin{equation*}
	\begin{split}
		\langle x^{i}\nabla_{i}\phi, \partial_{A,V}^{*}\partial_{A,V}\phi \rangle=&\langle x^{i}\nabla_{i}\phi, [\sqrt{-1}\Lambda_{\omega}F_{A},\phi]+\sqrt{-1}\Lambda_{\omega}F_{h_{V}}\phi \rangle\\
		=&-\langle x^{i}\nabla_{i}\phi, \frac{d\phi}{dt}\rangle-\langle x^{i}\nabla_{i}\phi, [[\phi,\phi^{*}],\phi]+\sqrt{-1}\Lambda_{\omega}F_{h_{V}}\phi \rangle.
	\end{split}
\end{equation*}
Note that, for any $\alpha\in\Omega^{1}(\mbox{End}(E))$, $\alpha^{*}=-\alpha$, we have
\begin{equation*}
	\begin{split}
		\Re\langle\alpha,(\bar{\partial}_{A}-\partial_{A})[\phi,\phi^{*}]]\rangle+2\Re\langle[\alpha,\phi],\partial_{A,V}\phi\rangle=0.
	\end{split}
\end{equation*}
Set $x\odot\nabla_{A,V}\phi=\frac{1}{2}x^{i}\nabla_{A,V;i}\phi$. Since $G_{u_{0}}>0$, we have
\begin{equation*}
	\begin{split}
		I_{1}+I_{2}=&4r\int_{T_{r}(u_{0})}|\partial_{A,V}\phi|^{2}f^{2}G_{u_{0}}dv_{g}dt \\
		&-4r\Re\int_{T_{r}(u_{0})}\langle d(f^{2}G_{u_{0}})\wedge x\odot F_{A},F_{A}\rangle dv_{g}dt\\
		&-8r\Re\int_{T_{r}(u_{0})}\langle d(f^{2}G_{u_{0}})x\odot\nabla_{A,V}\phi,\nabla_{A,V}\phi\rangle dv_{g}dt\\
		&-4r\Re\int_{T_{r}(u_{0})}\langle x\odot F_{A},\frac{d A}{d t}\rangle f^{2}G_{u_{0}}dv_{g}dt \\
		&-8r\Re\int_{T_{r}(u_{0})}\langle x\odot \nabla_{A,V}\phi,\frac{d \phi}{d t}\rangle f^{2}G_{u_{0}}dv_{g}dt \\
		&-8r\Re\int_{T_{r}(u_{0})}\langle x\odot \nabla_{A,V}\phi,[[\phi,\phi^{*}],\phi]+\sqrt{-1}\Lambda_{\omega}F_{h_{V}}\phi\rangle f^{2}G_{u_{0}}dv_{g}dt\\
		&-2r\Re\int_{T_{r}(u_{0})}\langle x^{i}F_{A}(\nabla_{j}\partial_{i},\partial_{k})dx^{j}\wedge dx^{k},F_{A}\rangle f^{2}G_{u_{0}}dv_{g}dt\\
		&-4r\Re\int_{T_{r}(u_{0})}\langle x^{i}\nabla_{A,V}\phi(\nabla_{j}\partial_{i})dx^{j},\nabla_{A,V}\phi\rangle f^{2}G_{u_{0}}dv_{g}dt\\
		&+8r\Re\int_{T_{r}(u_{0})}\langle x\odot F_{V}\phi,\nabla_{A,V}\phi\rangle f^{2}G_{u_{0}}dv_{g}dt.
	\end{split}
\end{equation*}

For the second term $I_3$, we have
\begin{equation*}
	\begin{split}
		I_{3}=2r\int_{T_{r}(u_{0})}(t-t_{0})\partial_{t}e_{2}(A,\phi)(x,t)f^{2}(x)G_{u_{0}}(x,t)\sqrt{\det{(g_{ij})}}(x)dxdt,
	\end{split}
\end{equation*}
where
\begin{equation*}
	\begin{split}
		\partial_{t}e_{2}(A,\phi)=2\Re\bigg(\bigg\langle D_{A}\bigg(\frac{\partial A}{\partial t}\bigg),F_{A}\bigg\rangle+2\bigg\langle\bigg[\frac{\partial A^{1,0}}{\partial t},\phi\bigg],\partial_{A,V}\phi\bigg\rangle+2\bigg\langle\partial_{A,V}\frac{\partial\phi}{\partial t},\partial_{A,V}\phi\bigg\rangle \bigg).
	\end{split}
\end{equation*}
So we obtain that
\begin{equation*}
	\begin{split}
		I_{3}=&-4r\Re\int_{T_{r}(u_{0})}(t-t_{0})\Big(\Big|\frac{\partial A}{\partial t}\Big|^{2}+2\Big|\frac{\partial \phi}{\partial t}\Big|^{2}\Big)f^{2}G_{u_{0}}dv_{g}dt\\
		&-4r\Re\int_{T_{r}(u_{0})}(t-t_{0})\bigg\langle d(f^{2}G_{u_{0}})\wedge\frac{\partial A}{\partial t},F_{A}\bigg\rangle dv_{g}dt\\
		&-8r\Re\int_{T_{r}(u_{0})}(t-t_{0})\bigg\langle d(f^{2}G_{u_{0}})\frac{\partial\phi}{\partial t},\nabla_{A,V}\phi\bigg\rangle dv_{g}dt\\
		&-8r\Re\int_{T_{r}(u_{0})}(t-t_{0})\bigg\langle \frac{\partial\phi}{\partial t},[[\phi,\phi^{*}],\phi]+\sqrt{-1}\Lambda_{\omega}F_{h_{V}}\phi\bigg\rangle f^{2}G_{u_{0}}dv_{g}dt.
	\end{split}
\end{equation*}
Note that $\partial_{i}G_{u_{0}}=\frac{x^{i}G_{u_{0}}}{2(t-t_{0})}$. Set
\begin{equation*}
	\begin{split}
		x\cdot F_{A}=\frac{1}{2}g^{ij}x^{j}F_{A,ik}dx^{k},&\ \ \ \ x\cdot \nabla_{A,V}\phi=\frac{1}{2}x^{j}g^{ij}\nabla_{A,V;i}\phi,\\
		\nabla f\cdot F_{A}=2g^{ij}f^{-1}\partial_{j}f F_{A;ik},&\ \ \ \ \nabla f\cdot \nabla_{A,V}\phi=2g^{ij}f^{-1}\partial_{j}f\nabla_{i}\phi.
	\end{split}
\end{equation*}
For any $\alpha\in\Omega^{1}(\mbox{End}(E))$, $\beta\in\Gamma(V\otimes\mbox{End}(E))$, we have
\begin{equation*}
	\begin{split}
		\langle d(f^{2}G_{u_{0}})\wedge \alpha,F_{A}\rangle
		=\langle \alpha, \nabla f\cdot F_{A}\rangle f^{2}G_{u_{0}}+\frac{1}{t-t_{0}}\langle \alpha, x\cdot F_{A}\rangle f^{2}G_{u_{0}},
	\end{split}
\end{equation*}
and
\begin{equation*}
	\begin{split}
		\langle d(f^{2}G_{u_{0}})\beta,\nabla_{A,V}\phi\rangle
		=\langle \beta, \nabla f\cdot \nabla_{A,V}\phi\rangle f^{2}G_{u_{0}}+\frac{1}{t-t_{0}}\langle \beta, x\cdot \nabla_{A,V}\phi\rangle f^{2}G_{u_{0}}.
	\end{split}
\end{equation*}
Combining the above inequalities, we have
\begin{equation*}
	\begin{split}
&I_{1}+I_{2}+I_{3}\\
&=4r\int_{T_{r}(u_{0})}\frac{1}{|t-t_{0}|}\Big||t-t_{0}|\frac{\partial A}{\partial t}-x\odot F_{A}\Big|^{2}f^{2}G_{u_{0}}dv_{g}dt\\
		&+4r\Re\int_{T_{r}(u_{0})}\frac{1}{|t-t_{0}|}\Big\langle x\cdot F_{A}-x\odot F_{A},x\odot F_{A}-|t-t_{0}|\frac{\partial A}{\partial t}\Big\rangle f^{2}G_{u_{0}}dv_{g}dt\\
		&+4r\Re\int_{T_{r}(u_{0})}\Big\langle |t-t_{0}|\frac{\partial A}{\partial t}-x\odot F_{A},\nabla f\cdot F_{A}\Big\rangle f^{2}G_{u_{0}}dv_{g}dt\\
		&+8r\int_{T_{r}(u_{0})}\frac{1}{|t-t_{0}|}\Big||t-t_{0}|\frac{\partial\phi}{\partial t}-x\odot \nabla_{A,V}\phi\Big|^{2}f^{2}G_{u_{0}}dv_{g}dt\\
		&+8r\Re\int_{T_{r}(u_{0})}\frac{1}{|t-t_{0}|}\Big\langle x\cdot \nabla_{A,V}\phi-x\odot\nabla_{A,V}\phi,x\odot \nabla_{A,V}\phi-|t-t_{0}|\frac{\partial \phi}{\partial t}\Big\rangle f^{2}G_{u_{0}}dv_{g}dt\\
		&+8r\Re\int_{T_{r}(u_{0})}\Big\langle |t-t_{0}|\frac{\partial\phi}{\partial t}-x\odot \nabla_{A,V}\phi,\nabla f\cdot \nabla_{A,V}\phi\Big\rangle f^{2}G_{u_{0}}dv_{g}dt\\
		&-8r\Re\int_{T_{r}(u_{0})}\langle x\odot \nabla_{A,V}\phi,[[\phi,\phi^{*}],\phi]+\sqrt{-1}\Lambda_{\omega}F_{h_{V}}\phi\rangle f^{2}G_{u_{0}}dv_{g}dt\\
		&-2r\Re\int_{T_{r}(u_{0})}\langle x^{i}F_{A}(\nabla_{j}\partial_{i},\partial_{k})dx^{j}\wedge dx^{k},F_{A}\rangle f^{2}G_{u_{0}}dv_{g}dt\\
		&-4r\Re\int_{T_{r}(u_{0})}\langle x^{i}\nabla_{A,V}\phi(\nabla_{j}\partial_{i})dx^{j},\nabla_{A,V}\phi\rangle f^{2}G_{u_{0}}dv_{g}dt\\
		&+8r\Re\int_{T_{r}(u_{0})}\langle x\odot F_{V}\phi,\nabla_{A,V}\phi\rangle f^{2}G_{u_{0}}dv_{g}dt\\
		&-8r\Re\int_{T_{r}(u_{0})}(t-t_{0})\Big\langle \frac{\partial \phi}{\partial t},[[\phi,\phi^{*}],\phi]+\sqrt{-1}\Lambda_{\omega}F_{h_{V}}\phi\Big\rangle f^{2}G_{u_{0}}dv_{g}dt\\
		&+4r\int_{T_{r}(u_{0})}|\partial_{A,V}\phi|^{2}f^{2}G_{u_{0}}dv_{g}dt.
	\end{split}
\end{equation*}
By the Lemma \ref{ll:1} and Proposition \ref{ppp:1}, we have
\begin{equation*}
	\begin{split}
		\Big|\frac{\partial \phi}{\partial t}\Big|^{2}\leq C,\ \  |\phi|^{2}\leq C,
	\end{split}
\end{equation*}
  where the constant $C$ depends on the initial data $(A_{0},\phi_{0})$ and $\sup_{X}|\sqrt{-1}\Lambda_{\omega}F_{h_{V}}|$. Since $r\leq R\leq i_{X}$. According to the Yang's inequality, we have
\begin{equation*}
	\begin{split}
		I_{1}+I_{2}+I_{3}\geq&-Cr\int_{T_{r}(u_{0})}\frac{1}{|t-t_{0}|}|x\cdot F_{A}-x\odot F_{A}|^{2} f^{2}G_{u_{0}}dv_{g}dt\\
		&-Cr\int_{T_{r}(u_{0})}|t-t_{0}||\nabla f\cdot F_{A}|^{2} f^{2}G_{u_{0}}dv_{g}dt\\
		&-Cr\int_{T_{r}(u_{0})}\frac{1}{|t-t_{0}|}|x\cdot \nabla_{A,V}\phi-x\odot\nabla_{A,V}\phi|^{2} f^{2}G_{u_{0}}dv_{g}dt\\
		&-Cr\int_{T_{r}(u_{0})}|t-t_{0}||\nabla f\cdot \nabla_{A,V}\phi|^{2} f^{2}G_{u_{0}}dv_{g}dt\\
		&-Cr\int_{T_{r}(u_{0})}|x|^{2}|\nabla_{A,V}\phi|^{2} f^{2}G_{u_{0}}dv_{g}dt\\
		&-2r\Re\int_{T_{r}(u_{0})}\langle x^{i}F_{A}(\nabla_{j}\partial_{i},\partial_{k})dx^{j}\wedge dx^{k},F_{A}\rangle f^{2}G_{u_{0}}dv_{g}dt\\
		&-4r\Re\int_{T_{r}(u_{0})}\langle x^{i}\nabla_{A,V}\phi(\nabla_{j}\partial_{i})dx^{j},\nabla_{A,V}\phi\rangle f^{2}G_{u_{0}}dv_{g}dt\\
		&-Cr,
	\end{split}
\end{equation*}
where the constant $C$ depends on the initial data $(A_{0},\phi_{0})$ and $\sup_{X}|\sqrt{-1}\Lambda_{\omega}F_{h_{V}}|$.

For the last term $I_4$, we have
\begin{equation*}
	\begin{split}
		I_{4}=&r\int_{T_{r}(u_{0})}e_{2}(A,\phi)x^{i}\partial_{i}(f^{2}\sqrt{\det{(g_{ij})}})G_{u_{0}}dxdt\\
		=&r\int_{T_{r}(u_{0})}e_{2}(A,\phi)2x^{i}f\partial_{i}(f)G_{u_{0}}dv_{g}dt+r\int_{T_{r}(u_{0})}e_{2}(A,\phi)f^{2}x^{i}\partial_{i}(\sqrt{\det{(g_{ij})}})G_{u_{0}}dxdt\\
		=&r\int_{T_{r}(u_{0})}e_{2}(A,\phi)2x^{i}f\partial_{i}(f)G_{u_{0}}dv_{g}dt+\frac{r}{2}\int_{T_{r}(u_{0})}e_{2}(A,\phi)x^{i}\tr(g^{-1}\partial_{i}g)f^{2}G_{u_{0}}dv_{g}dt.
	\end{split}
\end{equation*}
Since
\begin{equation*}
	|g_{ij}-\delta_{ij}|\leq C|x|^{2},\ \ |\partial_{i}g_{jk}|\leq C|x|,\ \ |\Gamma_{ij}^{k}|\leq C|x|,
\end{equation*}
then
\begin{equation*}
	\begin{split}
		|x\cdot F_{A}-x\odot F_{A}|^{2}&\leq C|x|^{6}|F_{A}|^{2},\\
		|x\cdot \nabla_{A,V}\phi-x\odot \nabla_{A,V}\phi|^{2}&\leq C|x|^{6}|\nabla_{A,V}\phi|^{2},\\
		\langle x^{i}F_{A}(\nabla_{k}\partial_{i},\partial_{j})dx^{j}\wedge dx^{k},F_{A}\rangle&\leq C|x|^{2}|F_{A}|^{2},\\
		\tr(g^{-1}\partial_{i}g)&\leq C|x|.
	\end{split}
\end{equation*}
Hence, we have
\begin{equation*}
	\begin{split}
		\frac{d\Phi(r)}{dr}\geq& -Cr\int_{T_{r}(u_{0})}\frac{|x|^{6}}{|t-t_{0}|}e_{2}(A,\phi)f^{2}G_{u_{0}}dv_{g}dt\\
		&-Cr\int_{T_{r}(u_{0})}|t-t_{0}||\nabla f\cdot F_{A}|^{2}f^{2}G_{u_{0}}dv_{g}dt\\
		&-Cr\int_{T_{r}(u_{0})}|t-t_{0}||\nabla f\cdot \nabla_{A,V}\phi|^{2} f^{2}G_{u_{0}}dv_{g}dt\\
		&-Cr\int_{T_{r}(u_{0})}|x|^{2}e_{2}(A,\phi) f^{2}G_{u_{0}}dv_{g}dt\\
		&-Cr\int_{T_{r}(u_{0})}f|\nabla f||x|e_{2}(A,\phi)G_{u_{0}}dv_{g}dt\\
		&-Cr.
	\end{split}
\end{equation*}
According to the Chen--Struwe's arguments in \cite{CS}, we know there exists a constant $\tilde{C}_{4}>0$ such that
\begin{equation*}
	\begin{split}
		&r^{-1}|t-t_{0}|\cdot |x|^{6}G_{u_{0}}\leq \tilde{C}_{4}(1+G_{u_{0}}),\\
		&r^{-1}|x|^{2}G_{u_{0}}\leq \tilde{C}_{4}(1+G_{u_{0}})
	\end{split}
\end{equation*}
on $T_{r}(u_{0})$. Then it follows that
\begin{equation*}
	\begin{split}
		&-Cr\int_{T_{r}(u_{0})}\Big(\frac{|x|^{6}}{|t-t_{0}|}+|t-t_{0}|+|x|^{2}\Big)e_{2}(A,\phi)f^{2}G_{u_{0}}dv_{g}dt\\
		&\geq -C\Phi(r)-Cr\textrm{YMH}(A_{0},\phi_{0}),
	\end{split}
\end{equation*}
According to the arguments in \cite[P. 1384]{NZ}, we have
\begin{equation*}
	\begin{split}
		-r\int_{T_{r}(u_{0})}|t-t_{0}|\cdot |\nabla f\cdot F_{A}|^{2}f^{2}G_{u_{0}}dv_{g}dt&\geq-\frac{C(n)r}{R^{2n}}\int_{P_{R}(u_{0})}|F_{A}|^{2}dv_{g}dt,\\
		-r\int_{T_{r}(u_{0})}|t-t_{0}|\cdot |\nabla f\cdot \nabla_{A,V}\phi|^{2}f^{2}G_{u_{0}}dv_{g}dt&\geq-\frac{C(n)r}{R^{2n}}\int_{P_{R}(u_{0})}|\nabla_{A,V}\phi|^{2}dv_{g}dt,\\
		-2r\int_{T_{r}(u_{0})}|x|\cdot |\nabla f|\cdot |f|\cdot e_{2}(A,\phi)G_{u_{0}}dv_{g}dt&\geq-\frac{C(n)r}{R^{2n}}\int_{P_{R}(u_{0})}e_{2}(A,\phi)dv_{g}dt.
	\end{split}
\end{equation*}
Combining the  above inequalities, we have
\begin{equation}\label{Teq2}
	\begin{split}
		\frac{d\Phi(r)}{dr}\geq -C\Phi(r)-Cr-\frac{Cr}{R^{2n}}\int_{P_{R}(u_{0})}e_{2}(A,\phi)dv_{g}dt,
	\end{split}
\end{equation}
where the constant $C$ depends on the geometry of $(X,\omega)$, $\sup_{X}|\sqrt{-1}\Lambda_{\omega}F_{h_{V}}|$ and the initial data $(A_{0},\phi_{0})$.
By integrating the above inequality (\ref{Teq2}) over $r$, we complete the proof.
\end{proof}

\begin{thm}\label{thm:5}
	Let $(A(t),\phi(t))$ be a solution of the Yang--Mills--Higgs flow (\ref{mymh}). There exist positive constants $\epsilon_{0},\delta_{0}<1/4$, such that if
	\begin{equation}
		R^{2-2n}\int_{P_{R}(u_{0})}e_{2}(A,\phi)dv_{g}dt<\epsilon_{0}
	\end{equation}
	holds for some $0<R\leq \min\{i_{X}/2,\sqrt{t_{0}}/2\}$, then for any $\delta\in(0,\delta_{0})$, we have
	\begin{equation}
		\sup_{P_{\delta R}(u_{0})}e_{2}(A,\phi)\leq\frac{16}{(\delta R)^{4}}.
	\end{equation}
\end{thm}
\begin{proof}
	For any $\delta\in(0,1/4]$,  we define the function
	\begin{equation*}
		\begin{split}
			f(r)=(2\delta R-r)^{4}\sup_{P_{r}(x_{0},t_{0})}e_{2}(A,\phi).
		\end{split}
	\end{equation*}
	Since $f(r)$ is continuous and $f(2\delta R)=0$, we know that $f(r)$ attains its maximum at some point $r_{0}\in[0,2\delta R)$. Suppose $(x_{1},t_{1})\in\bar{P}_{r_{0}}(x_{0},t_{0})$ is a point such that
	\begin{equation*}
		e_{2}(A,\phi)(x_{1},t_{1})=\sup_{P_{r}(x_{0},t_{0})}e_{2}(A,\phi).
	\end{equation*}
	We claim that  $f(r_{0})\leq 16$ when $\epsilon_{0},\delta_{0}$ are small enough. Otherwise, we have
	\begin{equation*}
		\rho_{0}:=e_{2}(A,\phi)(x_{1},t_{1})^{-1/4}=(2\delta R-r_{0})f(r_{0})^{-1/4}<\delta R-\frac{r_{0}}{2}.
	\end{equation*}
	Rescaling the Riemannian metric $\tilde{g}=\rho_{0}^{-2}g$, $\tilde{h}_{V}=\rho_{0}^{2}h_{V}$ and $t=t_{1}+\rho_{0}^{2}\tilde{t}$, we get
	\begin{equation*}
		\begin{split}
			&|F_{A}|_{\tilde{g}}^{2}=\rho_{0}^{4}|F_{A}|_{g}^{2},\\
			&|\partial_{A,V}\phi|_{\tilde{g},\tilde{h}_{V}}^{2}=\rho_{0}^{4}|\partial_{A,V}\phi|_{g,h_{V}}^{2}.
		\end{split}
	\end{equation*}
	Set
	\begin{equation*}
		\begin{split}
			&e_{\rho_{0}}(x,\tilde{t})=|F_{A}|_{\tilde{g}}^{2}+2|\partial_{A,V}\phi|_{\tilde{g},\tilde{h}_{V}}^{2}=\rho_{0}^{4}e_{2}(A,\phi)(x,t_{1}+\rho_{0}^{2}\tilde{t}),\\
			&\tilde{P}_{\tilde{r}}(x_{1},0)=B_{\rho_{0}\tilde{r}}(x_{1})\times[-\tilde{r}^{2},\tilde{r}^{2}].
		\end{split}
	\end{equation*}
Then we have $e_{\rho_{0}}(x_{1},0)=\rho_{0}^{4}e(A,\phi)(x_{1},t_{1})=1$, and
	\begin{equation*}
		\begin{split}
			\sup_{\tilde{P}_{1}(x_{1},0)}e_{\rho_{0}}&=\rho_{0}^{4}\sup_{P_{\rho_{0}}(x_{1},t_{1})}e(A,\phi)\leq \rho_{0}^{4}\sup_{P_{\delta R+r_{0}/2}(x_{0},t_{0})}e(A,\phi)\\
			&\leq \rho_{0}^{4}f(\delta R+r_{0}/2)(\delta R-r_{0}/2)^{-4}\leq 16.
		\end{split}
	\end{equation*}
	Thus
	\begin{equation}\label{eqn:315}
		\begin{split}
			|F_{A}|_{\tilde{g}}^{2}+2|\partial_{A,V}\phi|_{\tilde{g},\tilde{h}_{V}}^{2}\leq 16,\ \ \ \ on\ \ \tilde{P}_{1}(x_{1},0).
		\end{split}
	\end{equation}
	Combining above inequalities together with  the Proposition \ref{p:2} yields that
	\begin{equation*}
		\begin{split}
			\bigg(\frac{\partial}{\partial  \tilde{t}}-\Delta_{\tilde{g}}\bigg)&e_{\rho_{0}}=\rho_{0}^{6}\bigg(\frac{\partial}{\partial t}-\Delta_{g}\bigg)e_{2}(A,\phi)\\
			\leq &C_{1} \rho_{0}^{6}(|\phi|^{2}+|F_{A}|+|F_{h_V}|+|Ric|)|\partial_{A,V}\phi|^{2}+C_{2}\rho_{0}^{6}|\partial_{A}\sqrt{-1}\Lambda_{\omega}F_{h_{V}}||\phi||\partial_{A,V}\phi|\\
			&+C_{3}\rho_{0}^{6}(|F_{A}|+|\phi|^{2}+|Rm|)|F_{A}|^{2}+C_{4}\rho_{0}^{6}|F_{A}||\partial_{A,V}\phi|^{2}+C_{5}\rho_{0}^{6}|\phi|^{2}|F_{h_{V}}||F_{A}|\\
			\leq &C_{6} (e_{\rho_{0}}+\rho_{0}^{8})
		\end{split}
	\end{equation*}
	on $\tilde{P}_{1}(x_{1},0)$, where the constant $C_{6}$ depends only on the geometry of $(X,\omega)$, $F_{h_{V}}$ and $\sup_{X}|\phi_{0}|_{H_{0}}$. Then by the parabolic mean value inequality, we observe
	\begin{equation}\label{eqn:316}
		\begin{split}
			1<\sup_{\tilde{P}_{1/2}(x_{1},0)}(e_{\rho_{0}}+\rho_{0}^{8})&\leq C\int_{\tilde{P}_{1}(x_{1},0)}(e_{\rho_{0}}+\rho_{0}^{8})dv_{\tilde{g}}d\tilde{t}\\
			&=C_{7}\rho_{0}^{2-2n}\int_{P_{\rho_{0}}(x_{1},t_{1})}e_{2}(A,\phi)dv_{g}dt+C_{7}\rho_{0}^{8},\\
		\end{split}
	\end{equation}
	where the constant $C_{7}$ depends only on the geometry of $(X,\omega)$, $F_{h_{V}}$ and $\sup_{X}|\phi_{0}|_{H_{0}}$.
	
	 We choose normal geodesic coordinates centred at $x_{1}$, and let $f\in C^{\infty}_{0}(B_{R/2}(x_{1}))$ be a smooth cut-off function such that $0\leq f\leq 1$, $f\equiv 1$ on $B_{R/4}(x_{1})$, $|d f|\leq 8/R$ on $B_{R/2}(x_{1})\setminus B_{R/4}(x_{1})$. Taking $r_{1}=\rho_{0}$ and $r_{2}=\delta_{0} R$, and applying the monotonicity inequality, we obtain
	\begin{equation*}
		\begin{split}
			&\rho_{0}^{2-2n}\int_{P_{\rho_{0}}(x_{1},t_{1})}e_{2}(A,\phi)dv_{g}dt\\
			\leq& C\rho_{0}^{2}\int_{P_{\rho_{0}}(x_{1},t_{1})}e_{2}(A,\phi)G_{(x_{1},t_{1}+2\rho_{0}^{2})}f^{2}dv_{g}dt\\
			\leq& C\rho_{0}^{2}\int_{T_{\rho_{0}}(x_{1},t_{1}+2\rho_{0}^{2})}e_{2}(A,\phi)G_{(x_{1},t_{1}+2\rho_{0}^{2})}f^{2}dv_{g}dt\\
			\leq& C_{*}r_{2}^{2}\int_{T_{r_{2}}(x_{1},t_{1}+2\rho_{0}^{2})}e_{2}(A,\phi)G_{(x_{1},t_{1}+2\rho_{0}^{2})}f^{2}dv_{g}dt+C_{*}C_{10}\delta_{0}^{2}R^{2}\\
			&+C_{*}(R/2)^{2-2n}\int_{P_{R/2}(x_{1},t_{1})}e_{2}(A,\phi)dv_{g}dt\\
			\leq& C_{*}\delta_{0}^{2-2n}R^{2-2n}\int_{P_{R}(x_{0},t_{0})}e_{2}(A,\phi)dv_{g}dt+C_{*}C_{10}\delta_{0}^{2}R^{2}\\
			\leq& C_{8}(\delta_{0}^{2-2n}\epsilon_{0}+\delta_{0}^{2}R^{2}),
		\end{split}
	\end{equation*}
	where the constant $C_{8}$ depends only on the geometry of $(X,\omega)$, $F_{h_{V}}$ and $\sup_{X}|\phi_{0}|_{H_{0}}$. Choosing $\epsilon_{0},\delta_{0}$ small enough such that $\tilde{C}_{4}\tilde{C}_{5}(\delta_{0}^{2-2n}\epsilon_{0}+\delta_{0}^{2}R^{2})+\tilde{C}_{4}\delta_{0}^{8}R^{8}< 1$, then a contradiction occurs. So we have $f(r_{0})\leq 16$, which implies
	\begin{equation*}
		\sup_{P_{\delta R}(u_{0})}e_{2}(A,\phi)\leq 16/(\delta R)^{4}.
	\end{equation*}
\end{proof}

Using the above $\epsilon$-regularity theorem, and following the arguments of Hong and Tian (\cite{HT}) for the Yang--Mills flow case, we give the proof of the first part of Theorem \ref{thm:main1}.

\begin{proof}[Proof of Theorem \ref{thm:main1} (1)]
	By Proposition \ref{p:1}, for any $t_{k}\rightarrow+\infty$ and $a>0$, we have
	\begin{equation*}
		\begin{split}
			\int_{t_{k}-a}^{t_{k}+a}\int_{X}\bigg(\Big|\frac{\partial A}{\partial t}\Big|^{2}+\Big|\frac{\partial \phi}{\partial t}\Big|^{2}+\Big|\frac{\partial \phi^{*}}{\partial t}\Big|^{2}\bigg)dv_{g}dt\rightarrow 0,\ \ t_{k}\rightarrow+\infty.
		\end{split}
	\end{equation*}
Thus for any $\epsilon>0$, there is a constant $K$, when $k>K$ there holds that
	\begin{equation*}
		\begin{split}
			\int_{t_{k}-a}^{t_{k}+a}\int_{X}\bigg(\Big|\frac{\partial A}{\partial t}\Big|^{2}+\Big|\frac{\partial \phi}{\partial t}\Big|^{2}+\Big|\frac{\partial \phi^{*}}{\partial t}\Big|^{2})dv_{g}dt\leq \epsilon.
		\end{split}
	\end{equation*}
	Let
	\begin{equation*}
		\begin{split}
			\Sigma=\bigcap_{0<r<i_{X}}\{x\in X,\ \liminf_{k\rightarrow+\infty}r^{4-2n}\int_{B_{r}(x)}e_{2}(A,\phi)(\cdot,t_{k})dv_{g}\geq\epsilon_{1}\},
		\end{split}
	\end{equation*}
	where $\epsilon_{1}$ will be chosen later. For $x_{1}\in X\setminus\Sigma$, there exist $r_{1}$ and subsequence of $\{t_{k}\}$ (still denoted by $\{t_{k}\}$) such that
	\begin{equation*}
		r_{1}^{4-2n}\int_{B_{r_{1}}(x_{1})}e_{2}(A,\phi)(\cdot,t_{k})dv_{g}<\epsilon_{1}.
	\end{equation*}
	Let $s=t_{k}-r_{1}^{2}$, $\tau=t_{k}+r_{1}^{2}$ for any $t\in[s,\tau]$, using the Proposition \ref{p:3}, we have
	\begin{equation*}
		\begin{split}
			\int_{B_{r_{1}/2}(x_{1})}e_{2}(A,\phi)(\cdot, t)dv_{g}\leq& \int_{B_{r_{1}}(x_{1})}e_{2}(A,\phi)(\cdot,t_{k})dv_{g}\\
			&+2\int_{t_{k}-r_{1}^{2}}^{t_{k}+r_{1}^{2}}\int_{X}\bigg(\Big| \frac{\partial A}{\partial t}\Big|^{2}+\Big| \frac{\partial \phi}{\partial t}\Big|^{2}+\Big|\frac{\partial\phi^{*}}{\partial t}\Big|^{2})dv_{g}dt\\
			&+C_{1}\Big(\int_{t_{k}-r_{1}^{2}}^{t_{k}+r_{1}^{2}}\int_{X}\bigg(\Big| \frac{\partial A}{\partial t}\Big|^{2}+\Big| \frac{\partial \phi}{\partial t}\Big|^{2}+\Big|\frac{\partial\phi^{*}}{\partial t}\Big|^{2}\bigg)dv_{g}dt\Big)^{1/2}\\
			\leq &\int_{B_{r_{1}}(x_{1})}e_{2}(A,\phi)(\cdot,t_{k})dv_{g}+C_{2}\epsilon+C_{3}\epsilon^{1/2},
		\end{split}
	\end{equation*}
	where $C_{1}$, $C_{2}$ and $C_{3}$ are constants. So we have
	\begin{equation*}
		\begin{split}
			r_{1}^{2-2n}\int_{P_{r_{1}/2}(x_{1},t_{k})}e_{2}(A,\phi)dv_{g}dt\leq& r_{1}^{2-2n}\int_{t_{k}-(r_{1}/2)^{2}}^{t_{k}+(r_{1}/2)^{2}}\int_{B_{r_{1}}(x_{1})}e_{2}(A,\phi)(\cdot,t_{k})dv_{g}dt\\
			&+(C_{1}\epsilon+C_{2}\epsilon^{1/2})r_{1}^{2-2n}\int_{t_{k}-(r_{1}/2)^{2}}^{t_{k}+(r_{1}/2)^{2}}dt\\
			\leq &\frac{1}{2}r_{1}^{4-2n}\int_{B_{r_{1}}(x_{1})}e_{2}(A,\phi)(\cdot,t_{k})dv_{g}+\frac{1}{2}(C_{1}\epsilon+C_{2}\epsilon^{1/2})r_{1}^{4-2n},
		\end{split}
	\end{equation*}
  Choosing $\epsilon_{1}=\frac{\epsilon_{0}}{2^{2n-2}}$, and $\epsilon$ small enough such that
	\begin{equation*}
		\begin{split}
			(C_{1}\epsilon+C_{2}\epsilon^{1/2})r_{1}^{4-2n}2^{2n-3}\leq \epsilon_{0}/2,
		\end{split}
	\end{equation*}
	then we have
	\begin{equation*}
		\begin{split}
			\bigg(\frac{r_{1}}{2}\bigg)^{2-2n}\int_{P_{r_{1}/2}(x_{1},t_{k})}e_{2}(A,\phi)dv_{g}dt< \epsilon_{0}.
		\end{split}
	\end{equation*}
	By the $\epsilon$-regularity, we have
	\begin{equation}
		\sup_{P_{\delta_{0} r_{1}}(x_{1},t_{k})}e_{2}(A,\phi)\leq\frac{16}{(\delta_{0} r_{1})^{4}}.
	\end{equation}
	
	{\bf $\Sigma$ closed:} For any $x\in B_{\delta_{0}r_{1}}(x_{1})$, we choose $r_{x}$ small enough such that $B_{r_{x}}(x)\subset B_{\delta_{0}r_{1}}(x_{1})$ and
	\begin{equation*}
		\begin{split}
			r_{x}^{4-2n}\int_{B_{r_{x}}(x)}e_{2}(A,\phi)(\cdot,t_{k})dv_{g}\leq r_{x}^{4-2n}r_{x}^{2n}\frac{16}{(\delta_{0} r_{1})^{4}}\leq \frac{16 r_{x}^{4}}{(\delta_{0} r_{1})^{4}}<\epsilon_{1},
		\end{split}
	\end{equation*}
	that is, $B_{\delta_{0}r_{1}}(x_{1})\subset X\setminus \Sigma$. So $\Sigma$ is closed.
	
	{\bf $\mathcal{H}^{2n-4}(\Sigma)$:} Since $\Sigma$ is closed, for any $\delta>0$, there exist a finite number of geodesic balls $\{B_{r_{i}}(x_{i})\}$, $r_{i}<\delta$, such that $\{B_{r_{i}}(x_{i})\}$ is a cover of $\Sigma$, where $x_{i}\in\Sigma$ and $B_{r_{i}/2}(x_{i})\cap B_{r_{j}/2}(x_{j})=\emptyset$ for any $i\neq j$. Since $x_{i}\in\Sigma$, therefore
	\begin{equation*}
		\begin{split}
			r_{i}^{4-2n}\int_{B_{r_{i}/2}(x_{i})}e_{2}(A,\phi)(\cdot,t_{k})dv_{g}>2^{4-2n}\epsilon_{1}
		\end{split}
	\end{equation*}
	for sufficiently large $k$. That is,
	\begin{equation*}
		\begin{split}
			r_{i}^{2n-4}<2^{2n-4}\epsilon_{1}^{-1}\int_{B_{r_{i}/2}(x_{i})}e_{2}(A,\phi)(\cdot,t_{k})dv_{g},
		\end{split}
	\end{equation*}
	\begin{equation*}
		\begin{split}
			\sum_{i}r_{i}^{2n-4}<2^{2n-4}\epsilon_{1}^{-1}\int_{\cup_{i} B_{r_{i}/2}(x_{i})}e_{2}(A,\phi)(\cdot,t_{k})dv_{g}< +\infty.
		\end{split}
	\end{equation*}
	This implies that $\mathcal{H}^{2n-4}(\Sigma)< +\infty$.
	
	{\bf Convergence:} From the previous arguments, for any $x_{0}\in X\setminus\Sigma$, there exist $r_{0}$ and $\{t_{k}\}$ such that
	\begin{equation*}
		\begin{split}
			\sup_{P_{r_{0}}(x_{0},t_{k})}e_{2}(A,\phi)\leq C.
		\end{split}
	\end{equation*}
	By Uhlenbeck's weak compactness theorem, there exist a subsequence $\{t_{k^{'}}\}$ and gauge transformation $\{\sigma(k^{'})\}$ such that $\sigma(k^{'})\cdot(A(t_{k^{'}}),\phi(t_{k}^{'}))$ converges to a $V$-twisted Higgs pair $(A_{\infty},\phi_{\infty})$ on $(E_{\infty},H_{\infty})$ weakly in $W^{1,2}_{loc}(X\setminus\Sigma)$ and $(A_{\infty},\phi_{\infty})$ is a solution of the equation (\ref{YMHe}) outside $\Sigma$. By the standard parabolic estimates and using Hong--Tian's argument (Proposition 6 in [23]), we know that $\sigma(k^{'})\cdot (A(t_{k^{'}}),\phi(t_{k}^{'}))$ converges to $(A_{\infty},\phi_{\infty})$ in $C_{loc}^{\infty}$-topology outside $\Sigma$.
	
	{\bf Holomorphic Orthogonal split:} By the equation (\ref{YMHe}), we have
	\begin{equation*}
		\begin{split}
			D_{A_{\infty}}\theta_{\infty}=0,\ \ \ [\theta_{\infty},\phi_{\infty}]=0,
		\end{split}
	\end{equation*}
	where $\theta_{\infty}=\sqrt{-1}\Lambda_{\omega}F_{A_{\infty}}+[\phi_{\infty},\phi_{\infty}^{*H_{\infty}}]$. Since $\theta_{\infty}$ is parallel and $\theta_{\infty}^{*H_{\infty}}=\theta_{\infty}$,  we can decompose  $E_{\infty}$ and $\phi_{\infty}$  according to the eigenvalues of $\theta_{\infty}$. So we can obtain a holomorphic orthonogal decomposition
	\begin{equation*}
		\begin{split}
			E_{\infty}=\oplus_{i=1}^{l}E_{\infty}^{i},
		\end{split}
	\end{equation*}
	and
	\begin{equation*}
		\phi_{\infty}^{i}:E_{\infty}^{i}\rightarrow V\otimes E_{\infty}^{i}.
	\end{equation*}
	Let $H_{\infty}^{i}$, $\phi_{\infty}^{i}$be the restrict of $H_{\infty}$, $\phi_{\infty}$ to $E_{\infty}^{i}$, then $(A_{\infty}^{i},\phi_{\infty}^{i})$ is a  $V$-twisted Higgs pair on $(E_{\infty}^{i},H_{\infty}^{i})$ and satisfies
	\begin{equation*}
		\sqrt{-1}\Lambda_{\omega}F_{A_{\infty}^{i}}+[\phi_{\infty}^{i},(\phi_{\infty}^{i})^{*H_{\infty}^{i}}]=\lambda_{i}Id_{E_{\infty}^{i}}.
	\end{equation*}
	
	{\bf Extend to reflexive sheaf:}
	Because $\phi(t)$ is bounded, so we have
	\begin{equation*}
		\int_{X\setminus\Sigma}|F_{A_{\infty}}|^{2}_{H_{\infty}}dv_{g}\leq C.
	\end{equation*}
	Since $\mathcal{H}^{2n-4}(\Sigma)< +\infty$, $\phi_{\infty}$ is holomorphic and $C^{0}$ bounded, and every metric $H_{\infty}^{i}$ satisfies the Higgs--Hermitian--Einstein equation, from Theorem 2 in Bando and Siu's article (\cite{BS}), we know that every $(E_{\infty}^{i},\bar{\partial}_{A_{\infty}^{i}})$ can be extended to the whole $X$ as a reflexive sheaf (which is also denoted by $(E_{\infty}^{i},\bar{\partial}_{A_{\infty}^{i}})$ for simplicity), $\phi_{\infty}^{i}$ and $H_{\infty}^{i}$ can be smoothly extended over the place where the sheaf $(E_{\infty}^{i},\bar{\partial}_{A_{\infty}^{i}})$ is locally free.
\end{proof}

\begin{cor}\label{cor}
	Let $(A(t_{k}),\phi(t_{k}))$ be a sequence of $V$-twisted Higgs pairs along the Yang--Mills--Higgs flow with the limit $(A_{\infty},\phi_{\infty})$. Then
	\begin{itemize}
		\item[(1)] $\theta(A(t_{k}),\phi(t_{k}))\rightarrow \theta(A_{\infty},\phi_{\infty})$ strongly in $L^{p}$ as $k\rightarrow +\infty$ for all $1\leq p<+\infty$ and $\lim_{t\rightarrow +\infty}\|\theta(A(t),\phi(t))\|_{L^{2}}^{2}=\|\theta(A_{\infty},\phi_{\infty})\|_{L^{2}}^{2}$.
		\item[(2)] $\|\theta(A_{\infty},\phi_{\infty})\|_{L^{\infty}}\leq\|\theta(A(t_{k}),\phi(t_{k}))\|_{L^{\infty}}\leq \|\theta(A(t_{0}),\phi(t_{0}))\|_{L^{\infty}}$ for $0\leq t_{0}\leq t_{k}$.
	\end{itemize}
\end{cor}

\section{Isomorphism of the limit object and $Gr^{HNS}(E,\bar{\partial}_{A_{0}},\phi_{0})$}\label{sec:Is}

\subsection{ HNS filtration and $L^{p}$-$\delta$-approximate critical Hermitian metric of twisted Higgs bundles}
In this subsection, we show that the HN type of the limit $V$-twisted Higgs sheaf is consistent with the initial $V$-twisted Higgs bundle. The key to the proof is to obtain the existence of $L^{p}$-$\delta$-approximate critical Hermitian metric. First, let's recall the Harder--Narasimhan--Seshadri filtration of $V$-twisted Higgs bundle. The proof is almost the same as the one used in holomorphic bundles case (\cite[Sections 7.15, 7.17, 7.18]{Ko}).
\begin{lem}
	Let $(E,\bar{\partial}_{A},\phi)$ be a $V$-twisted Higgs bundle on K\"ahler manifold $(X, \omega)$. Then there is a filtration of $E$ by $\phi$-invariant coherent subsheaves
	\begin{equation}
		0=E_{0}\subset E_{1}\subset\cdots\subset E_{l}=E,
	\end{equation}
	called the Harder--Narasimhan filtration of $V$-twisted Higgs bundle $(E,\bar{\partial}_{A},\phi)$, such that $Q_{i}=
	E_{i}/E_{i-1}$ is torsion-free and semistable. Moreover, $\mu_{\omega}(Q_{i}) > \mu_{\omega}(Q_{i+1})$.
\end{lem}

\begin{lem}
	Let $(F,\phi_{F})$ be a semistable $V$-twisted Higgs sheaf on K\"ahler manifold $(X, \omega)$. Then there is a filtration of $F$ by $\phi_{F}$-invariant coherent subsheaves
	\begin{equation}
		0=F_{0}\subset F_{1}\subset\cdots\subset F_{l}=F,
	\end{equation}
	called the Seshadri filtration of $V$-twisted Higgs bundle $(F,\phi_{F})$, such that $Q_{i}=
	F_{i}/F_{i-1}$ is torsion-free and stable. Moreover, $\mu_{\omega}(Q_{i}) = \mu_{\omega}(F)$.
\end{lem}

\begin{prop}
	Let $(E,\bar{\partial}_{A},\phi)$ be a $V$-twisted Higgs bundle on K\"ahler manifold $(X, \omega)$. Then there is a double filtration $\{E_{i,j}\}$ of $E$ called $\phi$-invariant Harder--Narasimhan--Seshadri filtration, such that $\{E_{i}\}_{i=1}^{l}$ is the HN filtration, and $\{E_{i,j}\}_{j=1}^{l_{i}}$ is a Seshadri filtration of $E_{i}/E_{i-1}$. Let $Q_{i,j}=E_{i,j}/E_{i,j-1}$, the associated graded object $Gr^{HNS}(E,\bar{\partial}_{A},\phi)=\oplus_{i}^{l}\oplus_{j=1}^{l_{i}} Q_{i,j}$ is uniquely determined by the isomorphism class of $(E,\bar{\partial}_{A},\phi)$.
\end{prop}

\begin{defn}
	For a $V$-twisted Higgs bundle $(E,\bar{\partial}_{A},\phi)$ of rank $r$, construct a nonincreasing $r$-tuple of numbers
	\begin{equation}
		\vec{\mu}(E,\bar{\partial}_{A},\phi)=(\mu_{1},\cdots, \mu_{r})
	\end{equation}
	from the HN filtration by setting: $\mu_{i}=\mu_{\omega}(Q_{j})$, for $rank(E_{j-1})+1 \leq i \leq rank(E_{j})$. We call
	$\vec{\mu}(E,\bar{\partial}_{A},\phi)$ the Harder--Narasimhan type of $(E,\bar{\partial}_{A},\phi)$.
\end{defn}
Notice that $\vec{\mu}(E_{\infty},\bar{\partial}_{A_{\infty}},\phi_{\infty})=\frac{\mbox{Vol}(X,\omega)}{2\pi}\vec{\lambda}_{\infty}=\frac{\mbox{Vol}(X,\omega)}{2\pi}(\lambda_{1},\cdots,\lambda_{r})$. In the following, we assume $\mbox{Vol}(X,\omega)=2\pi$.

For any $\phi$-invariant subsheaf $S$ of $(E,\bar{\partial}_{A},\phi)$, let $H$ be a metric on $E$, $\pi_{S}^{H}:E\rightarrow E$ be the induced projection map of $E$ to $S$. It is not hard to get
\begin{equation}\label{eq:4.4}
	\deg_{\omega}(S)=\frac{1}{2\pi}\int_{X}\tr((\sqrt{-1}\Lambda_{\omega}F_{H}+[\phi,\phi^{*H}])\pi_{S}^{H})dv_{g}-\frac{1}{2\pi}\int_{X}(|\bar{\partial}_{A}\pi_{S}^{H}|^{2}+|[\phi,\pi_{S}^{H}]|^{2})dv_{g}.
\end{equation}
Let $\mathcal{G}$ be  the unitary gauge group of $E$,  and its Lie algebra is denoted by $\mathfrak{u}(E)$. The following proposition can be derived directly from equation (\ref{eq:4.4}).
\begin{prop}\label{prop4:5}
	Let $g_{j} \in \mathcal{G}^{\mathbb{C}}$ and $(A_{j},\phi_{j})= g_{j}\cdot(A_{0},\phi_{0})$ be a sequence of complex gauge equivalent $V$-twisted Higgs structure on complex vector bundle $E$ of rank $r$, $S$ be a $\phi_{0}$-invariant subsheaf of $(E,\bar{\partial}_{A_{0}},\phi_{0})$. Suppose $\sqrt{-1}\Lambda_{\omega}F_{A_{j}}+[\phi_{j},\phi_{j}^{*}]\rightarrow a$ in $L^{1}$ as $j\rightarrow+\infty$, where $a\in L^{1}(\sqrt{-1}\mathfrak{u}(E))$, and that the eigenvalues $\lambda_{1}\geq\cdots\geq\lambda_{r}$of $\frac{1}{2\pi}a$ are constant. Then $\deg_{\omega}(S)\leq\sum_{i\leq rank(S)}\lambda_{i}$.
\end{prop}

According to  (\ref{eq2:2}) in Lemma \ref{ll:1}, we have
\begin{equation}
	\begin{split}
		\sum_{i=1}^{r}\mu_{i}=\deg_{\omega}(E,\bar{\partial}_{A})=\deg_{\omega}(E_{\infty},\bar{\partial}_{A_{\infty}})=\sum_{i=1}^{r}\lambda_{i}.
	\end{split}
\end{equation}
Let $\{E_{i}\}_{i=1}^{l}$ be the HN filtration of $(E,\bar{\partial}_{A_{0}},\phi_{0})$. According to Corollary \ref{cor} and Proposition \ref{prop4:5}, we have
\begin{equation}
	\begin{split}
		\sum_{\alpha\leq rank(E_{i})}\mu_{\alpha}=\deg_{\omega}(E_{i})\leq \sum_{\alpha\leq rank(E_{i})}\lambda_{\alpha}.
	\end{split}
\end{equation}
That is,
\begin{equation}\label{e:HNT}
\vec{\mu}(E,\bar{\partial}_{A_{0}},\phi_{0})\leq \vec{\lambda}_{\infty}.
\end{equation}
\medskip

 Let $\mathfrak{u}(r)$ be the Lie algebra of unitary group $U(r)$. Fixing a real number $\alpha\geq 1$, for any $a\in\mathfrak{u}(r)$, we define the function $\varphi_{\alpha}:\mathfrak{u}(r) \to \mathbb{R}$ by
 \begin{equation*}
     \varphi_{\alpha}(a)=\sum_{j=1}^{r}|\lambda_{j}|^{\alpha},
 \end{equation*}
 where $\sqrt{-1}\lambda_j$ are the eigenvalues of $a$. It is easy to show that there is a family of smooth convex ad-invariant functions $\varphi_{\alpha,\rho}$, $0<\rho\leq1$, such that $\varphi_{\alpha,\rho}\rightarrow \varphi_{\alpha}$ uniformly on compact subsets of $\mathfrak{u}(r)$ as $\rho\to 0$. Therefore, from \cite[Prop. 12.16]{AB}, we know that $\varphi_{\alpha}$ is convex function. For any real number $N$, we define
\begin{equation}
	\HYM_{\alpha,N}(A,\phi)=\int_{X}\varphi_{\alpha}(\sqrt{-1}(\theta(A,\phi)+N\cdot \textmd{Id}_{E}))dv_{g}.
\end{equation}
For any $\vec{\mu}=(\mu_{1},\cdots,\mu_{r})$, set $\HYM_{\alpha,N}(\mu)=2\pi\varphi_{\alpha}(\sqrt{-1}(\vec{\mu}+N))$. Then for any smooth convex ad-invariant function $\varphi$, we have
\begin{equation}\label{eq}
	\bigg(\frac{\partial}{\partial t}-\Delta\bigg)\varphi(\sqrt{-1}(\theta(A(t),\phi(t))+N\cdot \textmd{Id}_{E}))\geq 0,
\end{equation}
whose proof can be found in \cite{DW,LZ}. Since we can approximate $\varphi_{\alpha}$ by smooth convex ad-invariant functions $\varphi_{\alpha,\rho}\rightarrow \varphi_{\alpha}$, by (\ref{eq}) we know that $\HYM_{\alpha,N}(A(t),\phi(t))$ is nonincreasing along the flow. By Corollary \ref{cor}, we can choose a sequence $t_{j}\rightarrow +\infty$, such that $\HYM_{\alpha,N}(A(t_{j}),\phi(t_{j}))\rightarrow \HYM_{\alpha,N} (A_{\infty},\phi_{\infty})$. Then we have
\begin{equation}\label{eq:j}
	\lim_{t\rightarrow +\infty}\HYM_{\alpha,N}(A(t),\phi(t))=\HYM_{\alpha,N}(A_{\infty},\phi_{\infty})
\end{equation}
for any $\alpha\geq 1$ and any $N$.

\begin{lem}[\cite{DW}] \label{l:DW1}
	The functional $a\rightarrow (\int_{X}\varphi_{\alpha}(a)dv_{g})^{1/\alpha}$ defines a norm on $L^{\alpha}(\mathfrak{u}(E))$ which is equivalent to the $L^{\alpha}$ norm.
\end{lem}
\begin{lem}[\cite{DW}] \label{l:DW2}
\begin{itemize}
		\item[(1)] If $\vec{\mu}\leq\vec{\lambda}$, then $\varphi_{\alpha}(\sqrt{-1}\vec{\mu})\leq\varphi_{\alpha}(\sqrt{-1}\vec{\lambda})$ for all $\alpha\geq 1$.
		\item[(2)] Assume $\mu_{r}\geq 0$ and $\lambda_{r}\geq 0$. If $\varphi_{\alpha}(\sqrt{-1}\vec{\mu})=\varphi_{\alpha}(\sqrt{-1}\vec{\lambda})$ for all $\alpha$ in some set $S\subset[1,+\infty)$ possessing a limit point, then $\vec{\mu}=\vec{\lambda}$.
	\end{itemize}
\end{lem}

Let  $\{E_{i}\}$ be the HN filtration of $E$. Given a Hermitian metric $H$ on $E$, we can  define an $L_{1}^{2}$-Hermitian endomorphism
\begin{equation*}
    \Psi^{HN}(E,\phi,H)=\sum_{i=1}^{l}\mu_{i}(\pi_{i}^{H}-\pi_{i-1}^{H}),
\end{equation*}
 where $\pi_{i}^{H}$ is the projection map of $E$ to $E_{i}$. More generally, given a  filtration $F=\{F_{i}\}$ of $E$ and real numbers $\{\mu_{i}\}_{i=1}^{l}$, we can also define an $L_{1}^{2}$-Hermitian endomorphism  $\Psi(F,(\mu_{1},\cdots,\mu_{l}),H)=\sum_{i=1}^{l}\mu_{i}(\pi_{i}^{H}-\pi_{i-1}^{H})$.
\begin{defn}
	Fix $0<p<+\infty$ and $\delta>0$. An $L^{p}$-$\delta$-approximate critical Hermitian metric on a $V$-twisted Higgs bundle $(E,\phi)$ is a smooth metric $H$ such that
	\begin{equation}
		\|\sqrt{-1}\Lambda_{\omega}F_{A_{H}}+[\phi,\phi^{*H}]-\Psi^{HN}(E,\phi,H)\|_{L^{p}}\leq\delta,
	\end{equation}
	where $A_{H}$ is the Chern connection determined by $(\bar{\partial}_{E},H)$.
\end{defn}

Let $\{E_{i,j}\}$ be the HNS filtration of the $V$-twisted Higgs bundle $(E,\bar{\partial}_{A},\phi)$. Set
\begin{equation}
	\Sigma_{alg}=\cup_{i,j}\{Sing(E_{i,j})\cup Sing(Q_{i,j})\}.
\end{equation}
 It is well known that $\Sigma_{alg}$ is a complex analytic subset of complex codimension at least two. We call it the singular set of the HNS filtration. Since the HNS filtration fails to be given by subbundles on the singular set $\Sigma_{alg}$, it makes difficult to do analysis.

When $\Sigma_{alg}=\emptyset$, it is easy to show the existence of $L^{\infty}$-$\delta$-approximate critical Hermitian metric for any $\delta>0$. In general case, Sibley use Hironaka's desingularisation theorem (\cite{Hir1,Hir2}) to resolve the singularities $\Sigma_{alg}$ and obtain a filtration by subbundles.

\begin{prop}[\cite{Si}]
	Let $0=E_{0}\subset E_{1}\subset\cdots\subset E_{l-1}\subset E_{l}=E$ be a filtration of a holomorphic vector bundle $E$ on a complex manifold $X$ by saturated subsheaves and let $Q_{i}=E_{i}/E_{i-1}$. Then there is a finite sequence of blow-ups along complex submanifolds of $X$ whose composition $\pi:\tilde{X}\rightarrow X$ enjoys the following properties. There is a filtration
	\begin{equation}
		0=\tilde{E}_{0}\subset\tilde{E}_{1}\subset\cdots\subset\tilde{E}_{l-1}\subset \tilde{E}_{l}=\tilde{E}
	\end{equation}
	by subbundles such that $\tilde{E}_{i}$ is the saturation of $\pi^{*}E_{i}$. If $\tilde{Q}_{i}=\tilde{E}_{i}/\tilde{E}_{i-1}$, then we have exact sequences:
	\begin{equation}
		0\rightarrow E_{i}\rightarrow \pi_{*}\tilde{E}_{i}\rightarrow T_{i}\rightarrow 0
	\end{equation}
	and
	\begin{equation}
		0 \rightarrow Q_{i}\rightarrow \pi_{*}\tilde{Q}_{i}\rightarrow T_{i}^{'}\rightarrow 0,
	\end{equation}
	where $T_{i}$ and $T_{i}^{'}$ are torsion sheaves supported on the singular sets of $E_{i}$ and $Q_{i}$, respectively, and furthermore $\pi_{*}\tilde{E}_{i}=E_{i}$ and $Q_{i}^{**}=(\pi_{*}\tilde{Q}_{i})^{**}$.
\end{prop}

Let $\phi\in \Gamma(\mbox{End}(E)\otimes V)$ be a twisted Higgs field on holomorphic bundle $(E,\bar{\partial}_{A})$, $\tilde{V}=\pi^{*}V$ and $\tilde{\phi}=\pi^{*}\phi\in \Gamma(\mbox{End}(\tilde{E})\otimes \tilde{V})$ be the pullback twisted Higgs field on $\tilde{E}$. If the filtration $\{E_{i}\}_{i=1}^{l}$ is by $\phi$-invariant subsheaves, then the filtration $\{\tilde{E}_{i}\}_{i=1}^{l}$ in the above proposition is by $\tilde{\phi}$-invariant subbundles. So, we have the following proposition.
\begin{prop}\label{pp:1}
	Let $\{E_{i,j}\}$ be the HNS filtration of a $V$-twisted Higgs bundle $(E,\bar{\partial}_{A},\phi)$ on complex manifold $X$ and let $Q_{i,j}=E_{i,j}/E_{i,j-1}$. Then there is a finite sequence of blow-ups along complex submanifolds of $X$ whose composition $\pi:\tilde{X}\rightarrow X$ enjoys the following properties. There is a filtration $\{\tilde{E}_{i,j}\}$ by $\tilde{\phi}$-subbundles such that $\tilde{E}_{i,j}$ is the saturation of $\pi^{*}E_{i,j}$, $\pi_{*}\tilde{E}_{i,j}=E_{i,j}$ and $Q_{i,j}^{**}=(\pi_{*}\tilde{Q}_{i,j})^{**}$, where $\tilde{\phi}=\pi^{*}\phi$.
\end{prop}

Since the blow-up $\tilde{X}$ is also K\"ahler, we have a family of K\"ahler metrics given by $\omega_{\epsilon}=\pi^{*}\omega+\epsilon\eta$ on it, where $\eta$ is a certain K\"ahler metric.

\begin{thm}[\cite{Si}]
	Let $\tilde{S}$ be a subsheaf (with torsion free quotient $\tilde{Q}$) of a holomorphic vector bundle $\tilde{E}$ on
	$\tilde{X}$, where $\pi:\tilde{X}\rightarrow X$ is given by a sequence of blow-ups along complex submanifolds of codim $\geq 2$. Then there is a uniform constant $C$ independent of $\tilde{S}$ such that the degrees of $\tilde{S}$ and $\tilde{Q}$ with respect to $\omega_{\epsilon}$ satisfy:
	\begin{equation}
		\deg_{\omega_{\epsilon}}(\tilde{S})\leq \deg_{\omega}(\pi_{*} \tilde{S})+\epsilon C,\ \ \text{and}\ \  \deg_{\omega_{\epsilon}}(\tilde{Q})\geq \deg_{\omega}(\pi_{*}\tilde{Q})-\epsilon C.
	\end{equation}
\end{thm}

\begin{cor}
	Let $(\tilde{E},\tilde{\phi})$ be a $\tilde{V}$-twisted Higgs bundle over $\tilde{X}$, $(E,\phi)$ be a $V$-twisted Higgs bundle on $X$ satisfy $E=\pi_{*}\tilde{E}$ and $\tilde{\phi}=\pi^{*}\phi$. If $(E,\phi)$ is $\omega$-stable, then there is an $\epsilon_{1}$ such that  $(\tilde{E},\tilde{\phi})$ is $\omega_{\epsilon}$-stable for all $0<\epsilon\leq \epsilon_{1}$. Let $\vec{\mu}_{\epsilon}(\tilde{E},\tilde{\phi})$ be the HN type of $(\tilde{E},\tilde{\phi})$ with respect to $\omega_{\epsilon}$, then $\vec{\mu}_{\epsilon}(\tilde{E},\tilde{\phi})\rightarrow\vec{\mu}(E,\phi)$ as $\epsilon\rightarrow 0$.
\end{cor}

\begin{lem}[\cite{Si}] \label{l:1}
	Let $(X, \omega)$ be a compact K\"ahler manifold of complex dimension $n$, and $\pi:\tilde{X}\rightarrow X$ be a blow-up along a smooth complex submanifold $\Sigma$ of complex codimension $k\geq2$. Let $\eta$ be a K\"ahler metric on $X$, and consider the family of K\"ahler metrics $\omega_{\epsilon}=\pi^{*}\omega+\epsilon\eta$, $0<\epsilon<\epsilon_{1}$. Then for any $\alpha$ and $\tilde{\alpha}$ such that $1<\alpha<1+\frac{1}{2(k-1)}$ and $\frac{\alpha}{1-2(k-1)(\alpha-1)}<\tilde{\alpha}<+\infty$. Let $s=\frac{\tilde{\alpha}}{\tilde{\alpha}-\alpha}$, we have $\frac{\eta^{n}}{\omega_{\epsilon}^{n}}\in L^{2(\alpha-1)s}(\tilde{X}, \eta)$, and the $L^{2(\alpha-1)s}$-norm of $\frac{\eta^{n}}{\omega_{\epsilon}^{n}}$ is uniformly bounded in $\epsilon$.
\end{lem}

\begin{lem}[\cite{Si}] \label{l:2}
	Let $\pi:\tilde{X}\rightarrow X$ be a blow-up along a smooth complex submanifold $\Sigma$ of complex codimension $k$ and the family of metrics $\omega_{\epsilon}$ be the same as in the previous lemma. Let $F$ be a $(1,1)$-form with values in $\mbox{End}(\tilde{E})$. Let $1<\alpha<1+\frac{1}{4k(k-1)}$ and $\frac{\alpha}{1-2(k-1)(\alpha-1)}<\tilde{\alpha}<1+\frac{1}{2(k-1)}$. Then there is a number $\kappa_{0}$ such that for any $0<\kappa\leq\kappa_{0}$, there exists a constant $C$ independent of $\epsilon$, $\epsilon_{1}$, and $\kappa$, and a constant $C(\kappa)$ such that:
	\begin{equation}
		\|\Lambda_{\omega_{\epsilon}}F\|_{L^{\alpha}(\tilde{X},\omega_{\epsilon})}\leq C(\|\Lambda_{\omega_{\epsilon_{1}}}F\|_{L^{\tilde{\alpha}}(\tilde{X},\omega_{\epsilon_{1}})}+\kappa\|F\|_{L^{2}(\tilde{X},\omega_{\epsilon_{1}})})+\epsilon_{1}C(\kappa)\|F\|_{L^{2}(\tilde{X},\omega_{\epsilon_{1}})}.
	\end{equation}
\end{lem}
Using Donaldson's argument, we can obtain the following proposition.
\begin{prop}
	Let $(E,\phi)$ be a $V$-twisted Higgs bundle on a smooth K\"ahler manifold $(X, \omega)$, and let $F=\{E_{i}\}_{i=1}^{l}$ be a filtration of $(E,\phi)$ by saturated subsheaves. Let $\pi:\tilde{X}\rightarrow X$ be a blow-up along smooth complex submanifold and the family of metrics $\omega_{\epsilon}$ be the same as in the previous lemma. Let $(\tilde{E},\tilde{\phi})$ be the pullback $\tilde{V}$-twisted Higgs bundle and $\tilde{F}=\{\tilde{E}_{i}\}_{i=1}^{l}=\{Sat_{\tilde{E}}(\pi^{*}E_{i})\}_{i=1}^{l}$ be the filtration of $(\tilde{E},\tilde{\phi})$. Suppose $\tilde{E}_{i}$ are subbundles, and $\tilde{Q}_{i}=\tilde{E}_{i}/\tilde{E}_{i-1}$ are $\omega_{\epsilon}$-stable for all $0<\epsilon\leq\epsilon_{*}$. Then for any $\delta>0$, $1\leq p\leq +\infty$ and $0<\epsilon\leq\epsilon_{*}$, there is a smooth Hermitian metric $\tilde{H}$ on $\tilde{E}$ such that
	\begin{equation}
		\|\sqrt{-1}\Lambda_{\omega_{\epsilon}}F_{\bar{\partial}_{\tilde{E}},\tilde{H}}+[\tilde{\phi},\tilde{\phi}^{*\tilde{H}}]-\Psi(\tilde{F},(\mu_{1,\epsilon},\cdots,\mu_{l,\epsilon}),\tilde{H})\|_{L^{p}}\leq\delta,
	\end{equation}
	where $\mu_{i,\epsilon}$ is the slope of quotient $\tilde{Q}_{i}$ with respect to the metric $\omega_{\epsilon}$.
\end{prop}

We also need the following additional propositions in next proof.
\begin{prop}\label{p:lp}
	Let $(E,\phi)$ be a $V$-twisted Higgs bundle on a smooth K\"ahler manifold $(X, \omega)$, and let $F=\{E_{i}\}_{i=1}^{l}$ be a filtration of $(E,\phi)$ by saturated subsheaves. Let $\pi:\tilde{X}\rightarrow X$ be a blow-up along smooth complex submanifold of complex codimension $k$ and the family of metrics $\omega_{\epsilon}$ be the same as in the previous lemma. Let $(\tilde{E},\tilde{\phi})$ be the pullback $\tilde{V}$-twisted Higgs bundle and $\tilde{F}=\{\tilde{E}_{i}\}_{i=1}^{l}=\{Sat_{\tilde{E}}(\pi^{*}E_{i})\}_{i=1}^{l}$ be the filtration of $(\tilde{E},\tilde{\phi})$. Suppose for any $\tilde{\delta}>0$ and any $0<\epsilon\leq\epsilon_{*}$, there is a smooth Hermitian metric $\tilde{H}$ on $\tilde{E}$ such that
	\begin{equation}\label{p:eq}
		\|\sqrt{-1}\Lambda_{\omega_{\epsilon}}F_{\bar{\partial}_{\tilde{E}},\tilde{H}}+[\tilde{\phi},\tilde{\phi}^{*\tilde{H}}]-\Psi(\tilde{F},(\mu_{1,\epsilon},\cdots,\mu_{l,\epsilon}),\tilde{H})\|_{L^{2}(\omega_{\epsilon})}\leq\tilde{\delta},
	\end{equation}	
	Then for any $\delta^{'}>0$ and any $1< p < 1+\frac{1}{4k(k-1)}$, there is a smooth Hermitian metric $H$ on $E$ such that
	\begin{equation}\label{p2:e}
		\|\sqrt{-1}\Lambda_{\omega}F_{\bar{\partial}_{E},H}+[\phi,\phi^{*H}]-\Psi(F,(\mu_{1},\cdots,\mu_{l}),H)\|_{L^{p}(\omega)}\leq\delta^{'},
	\end{equation}
	where $\mu_{i}$ is the $\omega$-slope of sheaf $Q_{i}$.
\end{prop}
\begin{proof}
	{\bf Step 1:}
	Let $\epsilon_{1}\in(0, \epsilon_{*})$, by the condition, we can choose a smooth metric $\tilde{H}_{1}$ satisfies
	(\ref{p:eq}) for $\epsilon_{1}$ and $\tilde{\delta}$ which will be chosen small enough later. For simplicity, we denote
	$\Theta_{1}=\sqrt{-1}F_{\bar{\partial}_{\tilde{E}},\tilde{H}_{1}}$. Then
	\begin{equation}\label{p:e1}
		\begin{split}
			\|\sqrt{-1}&\Lambda_{\omega_{\epsilon}}F_{\bar{\partial}_{\tilde{E}},\tilde{H}_{1}}+[\tilde{\phi},\tilde{\phi}^{*\tilde{H}_{1}}]-\Psi(\tilde{F},(\mu_{1},\cdots,\mu_{l}),\tilde{H}_{1})\|_{L^{p}(\omega_{\epsilon})}\\
			\leq& \Big\|\Lambda_{\omega_{\epsilon}}\{\Theta_{1}+\frac{\omega_{\epsilon_{1}}}{n}[\tilde{\phi},\tilde{\phi}^{*\tilde{H}_{1}}]-\frac{\omega_{\epsilon_{1}}}{n}\Psi(\tilde{F},(\mu_{1,\epsilon_{1}},\cdots,\mu_{l,\epsilon_{1}}),\tilde{H}_{1})\}\Big\|_{L^{p}(\omega_{\epsilon})}\\
			&+\Big\|\frac{1}{n}\Lambda_{\omega_{\epsilon}}(\omega_{\epsilon}-\omega_{\epsilon_{1}})\{[\tilde{\phi},\tilde{\phi}^{*\tilde{H}_{1}}]-\Psi(\tilde{F},(\mu_{1,\epsilon_{1}},\cdots,\mu_{l,\epsilon_{1}}),\tilde{H}_{1})\}\Big\|_{L^{p}(\omega_{\epsilon})}\\
			&+\|\Psi(\tilde{F},(\mu_{1},\cdots,\mu_{l}),\tilde{H}_{1})-\Psi(\tilde{F},(\mu_{1,\epsilon_{1}},\cdots,\mu_{l,\epsilon_{1}}),\tilde{H}_{1})\|_{L^{p}(\omega_{\epsilon})}
		\end{split}
	\end{equation}
Set
	\begin{equation*}
		\begin{split}
			\tilde{\Psi}_{\epsilon_{1}}=&[\tilde{\phi},\tilde{\phi}^{*\tilde{H}_{1}}]-\Psi(\tilde{F},(\mu_{1,\epsilon_{1}},\cdots,\mu_{l,\epsilon_{1}}),\tilde{H}_{1}),\\
			\Theta_{2}=&\Theta_{1}+\frac{\omega_{\epsilon_{1}}}{n}\tilde{\Psi}_{\epsilon_{1}},\ \ \Theta_{3}=(\omega_{\epsilon}-\omega_{\epsilon_{1}})\tilde{\Psi}_{\epsilon_{1}},
		\end{split}
	\end{equation*}
	where $\|\tilde{\Psi}_{\epsilon_{1}}\|_{L^{2}(\omega_{\epsilon_{1}})}$ is uniformly bounded in $\epsilon_{1}$. Applying Lemma \ref{l:2} to $\Theta_{i}$, $i=2,3$, we have
	\begin{equation}
		\begin{split}
			\|\Lambda_{\omega_{\epsilon}}\Theta_{i}\|_{L^{p}(\omega_{\epsilon})}\leq C(\|\Lambda_{\omega_{\epsilon_{1}}}\Theta_{i}\|_{L^{\tilde{p}}(\omega_{\epsilon_{1}})}+\kappa\|\Theta_{i}\|_{L^{2}(\omega_{\epsilon_{1}})})+\epsilon_{1}C(\kappa)\|\Theta_{i}\|_{L^{2}(\omega_{\epsilon_{1}})}.
		\end{split}
	\end{equation}
	Following  the arguments in Sibley's article (\cite[Page 35]{Si}), if we choose $\kappa$ and $\epsilon_{1}$  small enough, we have
	\begin{equation}
		\begin{split}
			\|\Lambda_{\omega_{\epsilon}}\Theta_{i}\|_{L^{p}(\omega_{\epsilon})}\leq \frac{\delta}{3}.
		\end{split}
	\end{equation}
	On the other hand, since $\mu_{i,\epsilon}\rightarrow \mu_{i}$ as $\epsilon \rightarrow 0$, we may choose $\epsilon_{1}$ small enough so that the third term in (\ref{p:e1}) is also smaller than $\delta/3$. Then for any $\delta>0$ and $1 < p<1+\frac{1}{4k(k-1)}$, we have \begin{equation}\label{p:ee}
		\|\sqrt{-1}\Lambda_{\omega_{\epsilon}}F_{\bar{\partial}_{\tilde{E}},\tilde{H}_{1}}+[\tilde{\phi},\tilde{\phi}^{*\tilde{H}_{1}}]-\Psi(\tilde{F},(\mu_{1},\cdots,\mu_{l}),\tilde{H}_{1})\|_{L^{p}(\omega_{\epsilon})}\leq\delta
	\end{equation}
	for any $0<\epsilon\leq\epsilon_{1}$, where $\mu_{i}$ is the slope of $Q_{i}$ with respect to the metric $\omega$.
	
	{\bf Step 2:} In order to obtain a smooth metric on $E$, we need to use a cut-off argument. Since $\Sigma$ is a smooth complex submanifold, the open set $\{(x,v)\in N_{\Sigma}|\ |v|< R\}$ in the normal bundle $N_{\Sigma}$ of $\Sigma$, is diffeomorphic to an open neighborhood $U_{R}$ of $\Sigma$ for $R$ sufficiently small. For any small $R$, we may choose a smooth cut-off function $\psi_{R}$  satisfying   $\text{supp}\psi_R\subset U_{R}$, $\psi_R=1$ on $U_{R/2}$, $0 \leq \psi_{R} \leq 1$, and furthermore $|\partial\psi_{R}|^{2}_{\omega}+|\partial\bar{\partial}\psi_{R}|_{\omega} \leq CR^{-2}$, where $C$ is a positive constant independent of $R$.
	
	Let $H_{D}$ be a smooth Hermitian metric on bundle $E$, and $\tilde{H}_{1}$ be the metric on $\tilde{E}$ such that (\ref{p:ee}) holds for all $0 < \epsilon \leq \epsilon_{1}$ where $\delta\leq \frac{\delta^{'}}{4}$. Note that $E$ is isomorphic to $\tilde{E}$ outsides $\Sigma$, we can define
	\begin{equation}
		H_{R}=(1-\psi_{R})\tilde{H}_{1}+\psi_{R}H_{D}
	\end{equation}
	on bundle $E$, $\tilde{H}_{R}=\pi^{*}H_{R}$ and $\tilde{H}_{D}=\pi^{*}H_{D}$ on bundle $\tilde{E}$.
	
	Set $\Theta(\tilde{H}_{R})=\sqrt{-1}F_{\bar{\partial}_{\tilde{E}},\tilde{H}_{R}}$,  then we have
	\begin{equation}
		\begin{split}
			\int_{\tilde{X}}|\Lambda&_{\omega_{\epsilon}}\Theta(\tilde{H}_{R})+[\tilde{\phi},\tilde{\phi}^{*\tilde{H}_{R}}]-\Psi(\tilde{F},(\mu_{1},\cdots,\mu_{l}),\tilde{H}_{R})|^{p}_{\tilde{H}_{R}}\frac{\omega_{\epsilon}^{n}}{n!}\\
			\leq&\int_{\pi^{-1}(U_{R/2})}|\Lambda_{\omega_{\epsilon}}\Theta(\tilde{H}_{D})+[\tilde{\phi},\tilde{\phi}^{*\tilde{H}_{D}}]-\Psi(\tilde{F},(\mu_{1},\cdots,\mu_{l}),\tilde{H}_{D})|_{\tilde{H}_{D}}^{p}\frac{\omega_{\epsilon}^{n}}{n!}\\
			&+\int_{\tilde{X}\setminus\pi^{-1}(U_{R})}|\Lambda_{\omega_{\epsilon}}\Theta(\tilde{H}_{1})+[\tilde{\phi},\tilde{\phi}^{*\tilde{H}_{1}}]-\Psi(\tilde{F},(\mu_{1},\cdots,\mu_{l}),\tilde{H}_{1})|_{\tilde{H}_{1}}^{p}\frac{\omega_{\epsilon}^{n}}{n!}\\
			&+C(p)\int_{\pi^{-1}(U_{R}\setminus U_{R/2})}|\Lambda_{\omega_{\epsilon}}(\Theta(\tilde{H}_{R})-\Theta(\tilde{H}_{1}))|_{\tilde{H}_{R}}^{p}\frac{\omega_{\epsilon}^{n}}{n!}\\
			&+C(p)\int_{\pi^{-1}(U_{R}\setminus U_{R/2})}|\Lambda_{\omega_{\epsilon}}\Theta(\tilde{H}_{1})+[\tilde{\phi},\tilde{\phi}^{*\tilde{H}_{R}}]-\Psi(\tilde{F},(\mu_{1},\cdots,\mu_{l}),\tilde{H}_{R})|_{\tilde{H}_{R}}^{p}\frac{\omega_{\epsilon}^{n}}{n!}.
		\end{split}
	\end{equation}
	For the first term, we have
	\begin{equation}
		\begin{split}
			\int_{\pi^{-1}(U_{R/2})}|\Lambda_{\omega_{\epsilon}}&\Theta(\tilde{H}_{D})+[\tilde{\phi},\tilde{\phi}^{*\tilde{H}_{D}}]-\Psi(\tilde{F},(\mu_{1},\cdots,\mu_{l}),\tilde{H}_{D})|_{\tilde{H}_{D}}^{p}\frac{\omega_{\epsilon}^{n}}{n!}\\
			\leq& C_{1}\int_{\pi^{-1}(U_{R/2})}	\Big(\frac{\eta^{n}}{\omega_{\epsilon}^{n}}\Big)^{p-1}\frac{\eta^{n}}{n!}\\
			&+C_{2}\int_{\pi^{-1}(U_{R/2})}|[\tilde{\phi},\tilde{\phi}^{*\tilde{H}_{D}}]-\Psi(\tilde{F},(\mu_{1},\cdots,\mu_{l}),\tilde{H}_{D})|_{\tilde{H}_{D}}^{p}\frac{\omega_{\epsilon}^{n}}{n!},
		\end{split}
	\end{equation}
	where $C_{1}$ and $C_{2}$ are constants independent of $\epsilon$ and $R$.
	
	For the third term, we have
	\begin{equation}
		\begin{split}
			|\Lambda_{\omega_{\epsilon}}(F_{\tilde{H}_{R}}-F_{\tilde{H}_{1}})|_{\tilde{H}_{R}}\leq \Big(\frac{C_{3}}{R^{2}}+C_{4}\Big)\frac{\eta^{n}}{\omega_{\epsilon}^{n}},
		\end{split}
	\end{equation}
	where $C_{3}$ and $C_{4}$ are independent of $\epsilon$ and $R$. Then by H\"older's inequality, we have
	\begin{equation}
		\begin{split}
			\int_{\pi^{-1}(U_{R}\setminus U_{R/2})}|\Lambda&_{\omega_{\epsilon}}(\Theta(\tilde{H}_{R})-\Theta(\tilde{H}_{1}))|_{\tilde{H}_{R}}^{p}\frac{\omega_{\epsilon}^{n}}{n!}\\
			\leq & \Big(\int_{\pi^{-1}(U_{R}\setminus U_{R/2})}\Big(\frac{\eta^{n}}{\omega_{\epsilon}^{n}}\Big)^{(1-p)s}\frac{\eta^{n}}{n!}\Big)^{\frac{1}{s}}\Big(\int_{\pi^{-1}(U_{R}\setminus U_{R/2})}\Big(\frac{C_{3}}{R^{2\tilde{p}}}+C_{4}\Big)\frac{\eta^{n}}{n!}\Big)^{\frac{p}{\tilde{p}}},
		\end{split}
	\end{equation}
	where $s$ and $\tilde{p}$ are constants as in Lemma \ref{l:1}. Hence
	\begin{equation}
		\begin{split}
			\int_{\pi^{-1}(U_{R}\setminus U_{R/2})}\Big(\frac{C_{3}}{R^{2\tilde{p}}}+C_{4}\Big)\frac{\eta^{n}}{n!}\leq C_{3}R^{2n-2\tilde{p}}+C_{4}R^{2n}.
		\end{split}
	\end{equation}
	Since $p<1+\frac{1}{4k(k-1)}$, $\frac{p}{1-2(k-1)(p-1)}<\frac{2k}{2k-1}(1+\frac{1}{4k(k-1)})\leq\frac{3}{2}$, we may choose $\tilde{p}<2$. We also have
	\begin{equation}
		\begin{split}
			\int_{\pi^{-1}(U_{R})}|[\tilde{\phi},\tilde{\phi}^{*\tilde{H}_{D}}]-\Psi(\tilde{F},(\mu_{1},\cdots,\mu_{l}),\tilde{H}_{D})|_{\tilde{H}_{D}}^{p}\frac{\omega_{\epsilon}^{n}}{n!}&\rightarrow 0,\\
			\int_{\pi^{-1}(U_{R})}|[\tilde{\phi},\tilde{\phi}^{*\tilde{H}_{R}}]-\Psi(\tilde{F},(\mu_{1},\cdots,\mu_{l}),\tilde{H}_{R})|_{\tilde{H}_{R}}^{p}\frac{\omega_{\epsilon}^{n}}{n!}&\rightarrow 0
		\end{split}
	\end{equation}
	as $R\rightarrow 0$, uniformly in $\epsilon$.
	
	By above formulas and Lemma \ref{l:1}, choosing $R$ small enough, we have
	\begin{equation}
		\begin{split}
			\int_{\tilde{X}}|\Lambda&_{\omega_{\epsilon}}\Theta(\tilde{H}_{R})+[\tilde{\phi},\tilde{\phi}^{*\tilde{H}_{R}}]-\Psi(\tilde{F},(\mu_{1},\cdots,\mu_{l}),\tilde{H}_{R})|^{p}_{\tilde{H}_{R}}\frac{\omega_{\epsilon}^{n}}{n!}\leq\delta^{'}
		\end{split}
	\end{equation}
	for all $0<\epsilon\leq \epsilon_{1}$. Now let $\epsilon\rightarrow 0$, we have
	\begin{equation}
		\begin{split}
			\int_{X}|\sqrt{-1}\Lambda_{\omega}F_{\bar{\partial}_{E},H_{R}}+[\phi,\phi^{*H_{R}}]-\Psi(F,(\mu_{1},\cdots,\mu_{l}),H_{R})|^{p}_{H_{R}}\frac{\omega^{n}}{n!}\leq\delta^{'}.
		\end{split}
	\end{equation}
\end{proof}

\begin{thm}\label{thm:1}
	Let $(E,\bar{\partial}_{A_{0}}, \phi_{0})$ be a $V$-twisted Higgs bundle on a smooth K\"ahler manifold $(X,\omega)$,
	and $(A(t), \phi(t))$ be the smooth solution of the Yang--Mills--Higgs flow on Hermitian vector bundle $(E,H_{0})$ with initial data $(A_{0}, \phi_{0})$. Suppose that for any $\delta^{'}>0$ and any $1< p < p_{0}$ there is a smooth metric $H$ on $(E,\bar{\partial}_{A_{0}}, \phi_{0})$ such that (\ref{p2:e}) holds, where $\vec{\mu}_{0}$ is the HN type of $(E, \bar{\partial}_{A_{0}}, \phi_{0})$. Let $(A_{\infty}, \phi_{\infty})$ be an Uhlenbeck limit of $(A_{t}, \phi_{t})$, and $(E_{\infty}, H_{\infty})$ be the corresponding Hermitian vector bundle defined away from $\Sigma_{an}$. Then
	\begin{equation}
		\begin{split}
			\HYM_{\alpha,N}(A_{\infty},\phi_{\infty})=\lim_{t\rightarrow\infty}\HYM_{\alpha,N}(A(t),\phi(t))=\HYM_{\alpha,N}(\vec{\mu}_{0})
		\end{split}
	\end{equation}
	for all $1< \alpha< p_{0}$ and all $N \in\mathbb{R}$; Moreover, the HN type of $(E,\bar{\partial}_{A_{0}}, \phi_{0})$ is equal to the HN type of $(E_{\infty}, \bar{\partial}_{A_{\infty}}, \phi_{\infty})$.
\end{thm}
\begin{proof}
The proof is  the same as \cite{LZ,LZ1}, and we write the main part of the proof here for the convenience of the reader. First, by triangle inequality and Lemma \ref{l:DW1}, we have
\begin{equation}
	\begin{split}
	|(\HYM&_{\alpha,N}(A_{0},\phi_{0},H))^{1/\alpha}-(\HYM_{\alpha,N}(\vec{\mu}_{0}))^{1/\alpha}|\\
	\leq &\big(\int_{X}|(\varphi_{\alpha}(\sqrt{-1}(\theta(A_{0},\phi_{0},H)+N\cdot \textmd{Id}_{E})))^{1/\alpha}-(\varphi_{\alpha}(\sqrt{-1}(\vec{\mu}_{0}+N)))^{1/\alpha}|^{\alpha}dv_{g}\big)^{1/\alpha}\\
	\leq &\big(\int_{X}\varphi_{\alpha}(\sqrt{-1}(\theta(A_{0},\phi_{0},H)-\Psi(F,(\mu_{1},\cdots,\mu_{l}),H)))dv_{g}\big)^{1/\alpha}\\
	\leq &C(\alpha)\|\sqrt{-1}\Lambda_{\omega}F_{\bar{\partial}_{A_{0}},H}+[\phi_{0},\phi_{0}^{*H}]-\Psi(F,(\mu_{1},\cdots,\mu_{l}),H)\|_{L^{\alpha}}.
	\end{split}
\end{equation}
Combining this with condition (\ref{p2:e}), we know that for any $\delta>0$ and any $1<\alpha<p_{0}$ there is a metric $H$ such that
\begin{equation}
	\begin{split}
		\HYM_{\alpha,N}(A_{0},\phi_{0},H)\leq \HYM_{\alpha,N}(\vec{\mu}_{0})+\delta.
	\end{split}
\end{equation}
 Since the image of the degree map on line bundles is discrete, for fixed $\alpha$ ($1<\alpha \leq p_{0}$) and fixed $N$, we can define $\delta_{0}>0$ such that
\begin{equation}
2\delta_{0}+\HYM_{\alpha,N}(\vec{\mu}_{0}) = \min\{\HYM_{\alpha,N}(\vec{\mu}):\HYM_{\alpha,N}(\vec{\mu})> \HYM_{\alpha,N}(\vec{\mu}_{0})\},
\end{equation}
where $\vec{\mu}$ runs over all possible HN types of $V$-twisted Higgs sheaves with rank $r$.

Let $H$ be a Hermitian metric on the complex bundle $E$, and $(A^{H}(t), \phi^{H}(t))$ be the solution of the Yang--Mills--Higgs flow on Hermitian vector bundle $(E, H)$ with initial pair $(A_{0}^{H}, \phi_{0})\in \mathcal{B}_{H}$ where $A_{0}^{H}$ is the Chern connection associated with $(\bar{\partial}_{A_{0}},H)$. Let $(A^{H}_{\infty}, \phi_{\infty})$ be an Uhlenbeck limit along the flow. We can choose the metric $H$ such that
\begin{equation}\label{e:b1}
	\HYM_{\alpha,N}(A_{0},\phi_{0},H)\leq \HYM_{\alpha,N}(\vec{\mu}_{0})+\delta_{0}.
\end{equation}
From (\ref{e:HNT}), (\ref{eq:j}) and Lemma \ref{l:DW1}, it follows that
\begin{equation}
	\HYM_{\alpha,N}(\vec{\mu}_{0})\leq \HYM_{\alpha,N}(A^{H}_{\infty},\phi_{\infty})\leq \HYM_{\alpha,N}(\vec{\mu}_{0})+\delta_{0}.
\end{equation}
Thus by the definition of $\delta_{0}$, we must have $\HYM_{\alpha,N}(A^{H}_{\infty}, \phi_{\infty}) = \HYM_{\alpha,N}(\vec{\mu}_{0})$. This shows that the result holds if the metric $H_{0}$ satisfies (\ref{e:b1}).

Let $\mathscr{H}(E)$ be the space of all smooth metrics on $E$. For any fixed $\delta$, we define a subset of $\mathscr{H}(E)$ denoted by $\mathscr{H}_{\delta}$ as follows:  $H\in \mathscr{H}_{\delta}$ if there exists  some $T \geq 0$ such that
\begin{equation}
\HYM_{\alpha,N}(A^{H}(t), \phi^{H}(t)) < \HYM_{\alpha,N}(\vec{\mu}_{0}) + \delta
\end{equation}
for all $t \geq T$. In our proof, we may assume that $0 < \delta \leq \frac{\delta_{0}}{2}$. It is easy to see that $\mathscr{H}_{\delta}$ is not empty. Following the argument in \cite[Lemma 4.3]{DW} (or \cite[Thoerem 5.13]{LZ}), we can show that $\mathscr{H}_{\delta}$ is both closed and open in $\mathscr{H}(E)$. Since $\mathscr{H}(E)$ is connected, we obtain that $\mathscr{H}(E)=\mathscr{H}_{\delta}$. From the previous discussion, we have
\begin{equation}\label{e:b2}
	\HYM_{\alpha,N}(A_{\infty}, \phi_{\infty})=\lim_{t\rightarrow\infty}\HYM_{\alpha,N}(A(t), \phi(t)) = \HYM_{\alpha,N}(\vec{\mu}_{0}).
\end{equation}
	
Let $\vec{\lambda}_{\infty}$ be the HN type of $(E_{\infty}, \bar{\partial}_{A_{\infty}}, \phi_{\infty})$, by (\ref{e:b2}), we have $\varphi_{\alpha}(\sqrt{-1}(\vec{\mu}_{0}+N))=\varphi_{\alpha}(\sqrt{-1}(\vec{\lambda}_{\infty}+N))$ for all $1<\alpha<p_{0}$ and all $N$. Choosing $N$ large enough, by Lemma \ref{l:DW2}, we get $\vec{\mu}_{0}=\vec{\lambda}_{\infty}$.
\end{proof}

Let $\{\pi_{i}\}_{i=1}^{l}$ be a HN filtration of $(E,A_{0},\phi_{0})$, the action $g^{j}$ produces HN filtration $\{\pi_{i}^{(j)}\}$ of $(E,A_{j},\phi_{j})$. Then
\begin{lem}\label{l:3}
	Let $(E, \bar{\partial}_{A_{0}}, \phi_{0})$ be a $V$-twisted Higgs bundle on a smooth K\"ahler manifold $(X, \omega)$, and
	satisfy the same assumptions as that in Theorem \ref{thm:1}.
	\begin{itemize}
		\item[(1)] Let $\{\pi_{i}^{\infty}\}$ be the HN filtration of $(E_{\infty},\bar{\partial}_{\infty}, \phi_{\infty})$, then there is a subsequence of HN filtration $\{\pi_{i}^{(j)}\}$ converges to a filtration $\{\pi_{i}^{\infty}\}$ strongly in $L^{p}\cap L^{2}_{1,loc}$ off $\Sigma_{an}$ for all $1\leq p<+\infty$.
		\item[(2)]Assume the $V$-twisted Higgs bundle $(E,\bar{\partial}_{A_{0}}, \phi_{0})$ is semistable and $\{E_{s,i}\}$ is the Seshadri filtration of $(E, \bar{\partial}_{A_{0}}, \phi_{0})$, then, after passing to a subsequence, $\{\pi_{s,i}^{(j)}\}$ converges to a filtration $\{\pi_{s,i}^{\infty}\}$ strongly in $L^{p}\cap L^{2}_{1,loc}$  off $\Sigma_{an}$ for all $1\leq p<+\infty$, the rank and degree of $\pi_{i}^{\infty}$ is equal to the rank and degree of $\pi_{s,i}^{(j)}$ for all $i$ and $j$.
	\end{itemize}
\end{lem}
\begin{proof}
	See (\cite{DW,LZ1}) for details.
\end{proof}

Combining the previous Lemma \ref{l:3}, Corollary \ref{cor} and the fact that
\begin{equation}
	\sqrt{-1}\Lambda_{\omega}F_{A_{\infty}}+[\phi_{\infty},(\phi_{\infty})^{*H_{\infty}}]=\Psi^{HN}(A_{\infty},\phi_{\infty},H_{\infty}),
\end{equation}
we have the following proposition.
\begin{prop}\label{p:l}
	Let $(E, \bar{\partial}_{A_{0}}, \phi_{0})$ be a $V$-twisted Higgs bundle on compact K\"ahler manifold $(X, \omega)$,
	and satisfy the same assumptions as that in Theorem \ref{thm:1}. Then for any $\delta> 0$ and $1\leq p<\infty$, $(E,A_{0},\phi_{0})$ has an $L^{p}$-$\delta$-approximate critical Hermitian metric.
\end{prop}

According to Proposition \ref{pp:1}, we can resolve the singularities $\Sigma_{alg}$ by blowing up finitely many times along complex submanifolds, that is, we have a sequence of blow-ups
\begin{equation}
	X_{m}\stackrel{\pi_{m}}\longrightarrow X_{m-1}\stackrel{\pi_{m-1}}\longrightarrow\cdots\stackrel{\pi_{1}}\longrightarrow X_{0}=X.
\end{equation}
Applying Proposition \ref{p:lp}, Theorem \ref{thm:1} and Proposition \ref{p:l} finitely many times, we end up with the following theorems.
\begin{thm}
	Let $(E, \bar{\partial}_{A_{0}}, \phi_{0})$ be a $V$-twisted Higgs bundle on compact K\"ahler manifold $(X, \omega)$.
	Then for any $\delta > 0$ and $1\leq p < +\infty$, $(E, \bar{\partial}_{A_{0}}, \phi_{0})$ has an $L^{p}$- $\delta$-approximate critical Hermitian metric.
\end{thm}
\begin{thm}
	Let $(A(t), \phi(t))$ be a smooth solution of the Yang--Mills--Higgs flow on the Hermitian vector bundle $(E, H_{0})$ with initial twisted Higgs pair $(A_{0}, \phi_{0})$, and $(A_{\infty}, \phi_{\infty})$ be a Uhlenbeck limit. Let $E_{\infty}$ denote the vector bundle obtained from $(A_{\infty},\phi_{\infty})$ as that in
	Theorem \ref{thm:main1}. Then the Harder--Narasimhan type of the extended reflexive $V$-twisted Higgs sheaf
	$(E_{\infty}, \bar{\partial}_{A_{\infty}}, \phi_{\infty})$ is same as that of the original $V$-twisted Higgs bundle $(E, \bar{\partial}_{A_{0}}, \phi_{0})$, that is, $\vec{\lambda}_{\infty}= \vec{\mu}_{0}$.
\end{thm}

\subsection{Construction of non-trivial holomorphic mappings}

Let $\{E_{i,j}\}$ be the HNS filtration of the $V$-twisted Higgs bundle $(E, \bar{\partial}_{A_{0}}, \phi_{0})$, the associated graded object $Gr^{HNS}(E,\bar{\partial}_{A_{0}},\phi_{0})$ is uniquely determined by the isomorphism class of $(\bar{\partial}_{A_{0}}, \phi_{0})$. Let $(E_{\infty},\bar{\partial}_{A_{\infty}},\phi_{\infty})$ be the limiting $V$-twisted Higgs sheaf. In this subsection, we want to show $Gr^{HNS}(E,\bar{\partial}_{A_{0}},\phi_{0})^{**}\simeq (E_{\infty},\bar{\partial}_{A_{\infty}},\phi_{\infty})$. The key is to construct non-zero holomorphic mappings from the subsheaves in the HNS filtration of $(E,\bar{\partial}_{A_{0}},\phi_{0})$ to  $(E_{\infty},\bar{\partial}_{A_{\infty}},\phi_{\infty})$.  This result can be proved by induction on the length of HNS filtration. The inductive hypotheses on a sheaf $Q$ are following:
\begin{itemize}
	\item[(1)] There is a sequence of $V$-twisted Higgs structures $(A^{Q}_{j}, \phi^{Q}_{j})$ on $Q$ such that $(A^{Q}_{j}, \phi^{Q}_{j})\rightarrow(A^{Q}_{\infty}, \phi^{Q}_{\infty})$ in $C^{\infty}_{loc}$ off $\Sigma_{alg}\cup\Sigma_{an}$;
	\item[(2)] $(A^{Q}_{j}, \phi^{Q}_{j})=g_{j}(A^{Q}_{0}, \phi^{Q}_{0})$ for some $g_{j}\in \mathcal{G}^{\mathbb{C}}(Q)$;
	\item[(3)] $(Q, \bar{\partial}_{A^{Q}_{0}}, \phi^{Q}_{0})$ and $(Q_{\infty}, \bar{\partial}_{A^{Q}_{\infty}}, \phi^{Q}_{\infty})$ extended to $X$ as reflexive $V$-twisted Higgs sheaves with the same HN type;
	\item[(4)] $\|\phi^{Q}_{j}\|_{C^{0}}$ and $\|\sqrt{-1}\Lambda_{\omega}(F_{A^{Q}_{j}})\|_{L^{1}(\omega)}$ is uniformly bounded in $j$.
\end{itemize}
The following proposition is crucial, and its proof is the same as that of Proposition 4.1 of (\cite{LZ1}).

\begin{prop}[\cite{LZ1}] \label{pr:1}
	Let $(X, \omega)$ be a K\"ahler manifold, $(E,\bar{\partial}_{A_{0}},\phi_{0})$ be a $V$-twisted Higgs sheaf on $X$ with Hermitian metric $H_{0}$, $S$ be a $V$-twisted Higgs subsheaf of $(E,\bar{\partial}_{A_{0}},\phi_{0})$, and $(A_{j},\phi_{j})=g_{j}(A_{0}, \phi_{0})$ be a sequence of $V$-twisted Higgs pairs on $E$, where $g_{j}$ is a sequence of complex gauge transformations. Suppose that there exits a sequence of blow-ups: $\pi_{i}: X_{i}\rightarrow X_{i-1}$, $i=1, \cdots, r$ (where $X_{0}=X$, every $\pi_{i}$ is a blow-up with non-singular center; denoting $\pi=\pi_{r}\circ\cdots\circ\pi_{1}$); such that $\pi^{*}E$ and $\pi^{*}S$ are bundles, the pulling back geometric objects $\pi^{*}(A_{0}, \phi_{0})$, $\pi^{*}g_{j}$ and $\pi^{*}H_{0}$ can be extended smoothly on the whole $X_{r}$. Assume that $(A_{j}, \phi_{j})$ converges to $(A_{\infty}, \phi_{\infty})$ outside a closed subset $\Sigma_{an}$ of Hausdorff complex codimension $2$, and $|\Lambda_{\omega}(F_{A_{j}})|_{H_{0}}$ is bounded uniformly in $j$ in $L^{1}(\omega_{0})$. Let $i_{0}:(S, \bar{\partial}_{A_{0}})\rightarrow (E, \bar{\partial}_{A_{0}})$ be the holomorphic inclusion, then there is a subsequence of $g_{j}\circ i_{0}$, up to rescale, converges to a non-zero holomorphic map $f_{\infty}: (S, \bar{\partial}_{A_{0}})\rightarrow(E_{\infty}, \bar{\partial}_{A_{\infty}})$ in $C^{\infty}_{loc}$ off $\Sigma\cup \Sigma_{an}$, and $f_{\infty}\circ \phi_{0}= \phi_{\infty}\circ f_{\infty}$, where $\Sigma$ is the singular set of $S$ and $E$.
\end{prop}

The proof of the following lemma is standard and can be found in \cite{Ko}.
\begin{lem}\label{l:m}
Let $(E_{1}, \phi_{1})$ and $(E_{2}, \phi_{2})$ be semistable $V$-twisted Higgs sheaves with $rank(E_{1})=rank(E_{2})$ and $\deg(E_{1})=\deg(E_{2})$, let $f:E_{1}\rightarrow E_{2}$ be a nonzero sheaf homomorphism satisfying $f\circ\phi_{1}=\phi_{2}\circ f$. If $(E_{1}, \phi_{1})$ is stable, then $f$ is an isomorphism.
\end{lem}

\begin{proof}[Proof of Theorem \ref{thm:main1} (2)]
	The proof is the same as in \cite[Sec. 5]{LZ1}, which we write here for the convenience of the reader.	
	
	Let $S=E_{1,1}$ be the first stable $V$-twisted Higgs subsheaf corresponding to the HNS filtration $\{E_{i,j}\}$ of $(E, \bar{\partial}_{A_{0}}, \phi_{0})$, $\pi:\tilde{X}\rightarrow X$ be the resolution of singularities $\Sigma_{alg}$. Then the filtration of $\tilde{E}=\pi^{*}E$ is given by subbundles $\tilde{E}_{i,j}$, isomorphic to $E_{i,j}$ off $\tilde{\Sigma}=\pi^{-1}(\Sigma_{alg})$. Setting $(\tilde{A}_{j}, \tilde{\phi}_{j})=\pi^{*}(A_{j}, \phi_{j})$ and $\tilde{g}_{j}=\pi^{*}g_{j}$, then we have $(\tilde{A}_{j},\tilde{\phi}_{j})=\tilde{g}_{j}(\tilde{A}_{0},\tilde{\phi}_{0})$. By Theorem \ref{thm:main1}, we know that $(\tilde{A}_{j}, \tilde{\phi}_{j})$ converges to $(\tilde{A}_{\infty}, \tilde{\phi}_{\infty})$ in $C^{\infty}_{loc}$ topology outside $\pi^{-1}(\Sigma_{alg}\cup\Sigma_{an})$. By Corollary \ref{cor} and the uniform $C^{0}$ bound on $\phi(t)$, we have $\|\sqrt{-1}\Lambda_{\omega_{0}}(F_{\tilde{A}_{j}})\|_{L^{\infty}}$, specially $\|\sqrt{-1}\Lambda_{\omega_{0}}(F_{\tilde{A}_{j}})\|_{L^{1}}$ is uniformly bounded in $j$, where $\omega_{0}=\pi^{*}\omega$.
	
	Using Proposition \ref{pr:1}, we have a non-zero smooth $\tilde{\phi}$-invariant holomorphic map $\tilde{f}_{\infty}:\tilde{S} \rightarrow\tilde{E}_{\infty}$ off $\pi^{-1}(\Sigma_{alg}\cup\Sigma_{an})$. Since $\tilde{S}$ is isomorphic to $S$ off $\tilde{\Sigma}$, then we obtain a non-zero smooth $\phi$-invariant holomorphic map $f_{\infty}:S\rightarrow (E_{\infty},\bar{\partial}_{A_{\infty}})$ on $X\setminus\Sigma_{an}\cup\Sigma_{alg}$. By Hartog's theorem, $f_{\infty}$ extends to a $V$-twisted Higgs sheaf homomorphism $f_{\infty}:(S, \phi_{0})\rightarrow (E_{\infty}, \bar{\partial}_{A_{\infty}}, \phi_{\infty})$ on $X$.
	
	Let $\pi^{(j)}_{1}$ denotes the projection to $g_{j}(S)$. By the Lemma \ref{l:3}, we know that $\pi_{1}^{(j)}\rightarrow \pi_{1}^{\infty}$ in $L^{p}\cap L_{1,loc}^{2}$ off $\Sigma_{an}$ and $\pi^{\infty}_{1}$ determines a $V$-twisted Higgs subsheaf $E^{\infty}_{1,1}$ of $(E_{\infty}, \bar{\partial}_{A_{\infty}}, \phi_{\infty})$, with $rank(E^{\infty}_{1,1})=rank(S)$ and $\mu(E^{\infty}_{1,1})=\mu(S)$. Since $(E_{\infty}, \bar{\partial}_{A_{\infty}}, \phi_{\infty})$ and $(E_{0}, \bar{\partial}_{A_{0}}, \phi_{0})$ have the same HN type, thus we have the $V$-twisted Higgs subsheaf $(E^{\infty}_{1,1}, \phi_{\infty})$ is semistable and
	\begin{equation}
		f_{\infty}:S\rightarrow E^{\infty}_{1,1}.
	\end{equation}
	Recall that $S=E_{1,1}$ is $V$-twisted Higgs stable. By the Lemma \ref{l:m}, we see that the non-zero holomorphic map $f_{\infty}$ must be injective on $X\setminus\Sigma_{an}\cup\Sigma_{alg}$ and $E^{\infty}_{1,1}$ must be a stable $V$-twisted Higgs subsheaf of $(E_{\infty}, \bar{\partial}_{A_{\infty}}, \phi_{\infty})$.
	
	Let $\{e_{\alpha}\}$ be a local frame of $S$, and $H_{j,\alpha\bar{\beta}}=\langle g_{j}(e_{\alpha}), g_{j}(e_{\beta}) \rangle_{H_{0}}$. We can write the orthogonal projection $\pi^{(j)}_{1}$ as
	\begin{equation}
		\pi^{(j)}_{1}(X)=\langle X, g_{j}(e_{\beta})\rangle_{H_{0}} H_{j}^{\alpha\bar{\beta}}g_{j}(e_{\alpha})
	\end{equation}
	for any $X\in \Gamma(E)$, where $(H_{j}^{\alpha\bar{\beta}})$ is the inverse of the matrix $(H_{j,\alpha\bar{\beta}})$. Because $g_{j}\circ i_{0}\rightarrow f_{\infty}$ in $C^{\infty}(\Omega)$, and $f_{\infty}$ is injective on $X\setminus\Sigma_{an}\cup\Sigma_{alg}$, then we can prove that $\pi^{(j)}_{1}\rightarrow \pi^{\infty}_{1}$ in $C^{\infty}_{loc}$ off $\Sigma_{an}\cup\Sigma_{alg}$.
	
	Let $Q=E/S$, then we have $Gr^{HNS}(E, \bar{\partial}_{A_{0}}, \phi_{0})=S\oplus Gr^{HNS}(Q, \bar{\partial}_{A^{Q}_{0}}, \phi^{Q}_{0})$. Write the orthogonal holomorphic decomposition $(E^{\infty}, \bar{\partial}_{A_{\infty}}, \phi_{\infty})=E^{\infty}_{1}\oplus Q_{\infty}$, where $Q_{\infty}=(E^{\infty}_{1})^{\perp}$. Using Lemma 5.12 in \cite{Da}, we can choose a sequence of unitary gauge transformation $u_{j}$ such that $\pi^{(j)}_{1}=u_{j}\tilde{\pi}_{j}u_{j}^{-1}$ where $\tilde{\pi}_{j}(E)=\pi^{\infty}_{1}(E)=E^{\infty}_{1}$ and $u_{j}\rightarrow Id_{E}$ in $C^{\infty}_{loc}$ on $X\setminus (\Sigma_{alg}\cup\Sigma_{an})$. It is easy to check that
	$u_{j}(Q_{\infty})=(\pi^{(j)}_{1}(E))^{\perp}$. Let $p:Q \rightarrow S^{\perp}$ be the $C^{\infty}$ bundle isomorphism outside singularity set. Noting the unitary gauge transformation $u_{0}:Q_{\infty}\rightarrow S^{\perp}$, and considering the induced $V$-twisted Higgs pair on $Q$, defined by
	\begin{equation}
	\begin{split}
		D_{A_{j}^{Q}}=&p^{-1}\circ u_{0}\circ u_{j}^{-1}\circ \pi_{j}^{\perp}\circ D_{A_{j}} \circ \pi_{j}^{\perp} \circ	u_{j}\circ u_{0}^{-1}\circ p,\\
		\phi_{j}^{Q}=p^{-1}\circ u_{0}&\circ u_{j}^{-1}\circ \pi_{j}^{\perp}\circ \phi_{j} \circ \pi_{j}^{\perp} \circ u_{j}\circ u_{0}^{-1}\circ p\in\Gamma(V\otimes \mbox{End}(Q)).
		\end{split}
	\end{equation}
	Let
	\begin{equation}
		h_{j}=p^{-1}\circ u_{0}\circ u_{j}^{-1}\circ \pi_{j}^{\perp}\circ g_{j}\circ \pi_{0}^{\perp}\in \mathcal{G}^{\mathbb{C}}(Q).
	\end{equation}
	Then, we have
	\begin{equation}
		\begin{split}
			\bar{\partial}_{A_{j}^{Q}}=&h_{j}\circ \bar{\partial}_{A_{0}^{Q}} \circ h_{j}^{-1},\\
			\partial_{A_{j}^{Q}}=&(h_{j}^{*})^{-1}\circ \partial_{A_{0}^{Q}}\circ h_{j}^{*},
		\end{split}
	\end{equation}
	\begin{equation}
		\begin{split}
			\phi_{j}^{Q}= h_{j}\circ \phi_{0}^{Q} \circ h_{j}^{-1},\\
		\end{split}
	\end{equation}
	and
	\begin{equation}
		\bar{\partial}_{A_{j}^{Q}}\phi_{j}^{Q}=0,\ \ \phi_{j}^{Q}\wedge \phi_{j}^{Q}=0,
	\end{equation}
	where we have used $h_{j}^{-1}=\pi_{0}^{\perp} \circ g_{j}^{-1}\circ\pi_{j}^{\perp}\circ u_{j}\circ u_{0}^{-1}\circ p$. On the other hand, by the definition, it is easy to check that $u_{0}^{*}p(A_{j}^{Q}, \phi_{j}^{Q})\rightarrow (A_{\infty}^{Q_{\infty}}, \phi_{\infty}^{Q_{\infty}})$ in $C^{\infty}_{loc }$. Now we check the third statement in the inductive hypotheses. Let's consider the Gauss--Codazzi equation on $(\pi_{1}^{(j)}(E))^{\perp}\simeq Q^{j}$,
	\begin{equation}
		\begin{split}
			F_{A_{Q^{j}}}=&(\pi_{1}^{(j)})^{\perp}\circ F_{A_{j}}\circ (\pi_{1}^{(j)})^{\perp} +\partial_{A_{j}}\pi_{1}^{(j)}\wedge \bar{\partial}_{A_{j}}\pi_{1}^{(j)},\\
			[\phi_{j}^{Q},(\phi_{j}^{Q})^{*}]=&(\pi_{1}^{(j)})^{\perp}\circ[\phi_{j},(\phi_{j})^{*}]\circ (\pi_{1}^{(j)})^{\perp} +[\pi_{1}^{(j)},\phi_{j}^{*}]\circ[\phi_{j},\pi_{1}^{(j)}],
		\end{split}
	\end{equation}
	where $D_{A_{Q^{j}}}=(\pi_{1}^{(j)})^{\perp}\circ D_{A_{j}}\circ (\pi_{1}^{(j)})^{\perp}$. Setting the $V$-twisted Higgs field $\phi_{Q^{j}}=(\pi_{1}^{(j)})^{\perp}\circ \phi_{j}\circ (\pi_{1}^{(j)})^{\perp}$,  we have
	\begin{equation}
		\begin{split}
			\int_{X}|\sqrt{-1}&\Lambda_{\omega}F_{A_{j}^{Q}}+[\phi_{j}^{Q} , (\phi_{j}^{Q})^{*}]-\Psi^{HN}(A_{j}^{Q}, \phi_{j}^{Q}, H_{0})|\frac{\omega^{n}}{n!}\\
			=&\int_{X}|\sqrt{-1}\Lambda_{\omega}F_{A_{Q^{j}}}+[\phi_{Q^{j}} , (\phi_{Q^{j}})^{*}]-\Psi^{HN}(A_{Q^{j}}, \phi_{Q^{j}}, H_{0})|\frac{\omega^{n}}{n!}\\
			=&\int_{X}|(\pi_{1}^{(j)})^{\perp}\{\sqrt{-1}\Lambda_{\omega}F_{A_{j}}+[\phi_{j} , (\phi_{j})^{*}
			]-\Psi^{HN}(A_{j}, \phi_{j}, H_{0})\}(\pi_{1}^{(j)})^{\perp}\\
			&+\sqrt{-1}\Lambda_{\omega }(\partial_{A_{j}}\pi_{1}^{(j)}\wedge \bar{\partial}_{A_{j}}\pi_{1}^{(j)})+[\pi_{1}^{(j)},\phi_{j}^{*}]\circ[\phi_{j},\pi_{1}^{(j)}]|\frac{\omega^{n}}{n!}\\
			\leq &  \int_{X}|\sqrt{-1}\Lambda_{\omega}F_{A_{j}}+[\phi_{j},(\phi_{j})^{*}]-\Psi^{HN}(A_{j}, \phi_{j}, H_{0})| +| \bar{\partial}_{A_{j}}\pi_{1}^{(j)}|^{2}\\
			&+|[\phi_{j}, \pi_{1}^{(j)}]|^{2}\frac{\omega^{n}}{n!}\\
			 \rightarrow & 0.
		\end{split}
	\end{equation}
	Since  $C^{0}$ norm of $\phi_{j}$ is uniformly bounded, then $\| \phi_{j}^{Q}\|_{C^{0}}$ and $\|\sqrt{-1}\Lambda _{\omega }(F_{A_{ j}^{Q}})\|_{L^{1}(\omega)}$ is uniformly bounded  in $j$.  So, $(Q , A_{j}^{Q} , \phi_{j}^{Q})$ satisfy the inductive hypotheses. Since we can resolve the singularity set $\Sigma_{alg}$ by blowing up finitely many times with non-singular center, and the pulling back of the HNS filtration is given by subbundles. The sheaf $Q$  and every geometric objects which we considered  are induced by the HNS filtration, so their pulling back are all smooth. Using Proposition \ref{pr:1} again,  by induction  we have
	\begin{equation}
		E_{\infty}\simeq Gr^{HNS
		}(E, \bar{\partial}_{A_{0}},\phi_{0})=\oplus_{i=1}^{l}\oplus_{j=1}^{l_{i}}Q_{i,j}
	\end{equation}
	on $X\setminus (\Sigma_{alg}\cup \Sigma_{an})$. By Theorem \ref{thm:main1}, we know that  $(E_{\infty}, \bar{\partial}_{A_{\infty}}, \phi_{\infty})$ can be extended to the whole $X$ as a reflexive $V$-twisted Higgs sheaf. By the uniqueness of reflexive extension in \cite{Siu}, we know that there exists a sheaf isomorphism
	\begin{equation}
		f:  Gr^{HNS
		}(E , \overline{\partial }_{A_{0}} , \phi_{0})^{**}\rightarrow   (E_{\infty}, \bar{\partial }_{A_{\infty}}, \phi_{\infty})
	\end{equation}
	on $X$. This completes the proof.
\end{proof}

\medskip

\noindent {\bf  Acknowledgements:} The authors thank Prof. Xi Zhang for his consistent encouragement. The research was supported by the National Key R and D Program of China 2020YFA0713100. The first and second authors are partially supported by the Natural Science Foundation of China [Grant Numbers 12141104 and 11721101]. The third author is supported by the Natural Science Foundation of China [Grant Number 12201001], the Natural Science Foundation of Anhui Province [Grant Number 2108085QA17], the Natural Science Foundation of Universities of Anhui Province [Grant Number KJ2020A0009]

\medskip

\noindent {\bf Conflicts of Interest:} The authors declare no conflicts of interest.

\end{document}